\documentclass[pdflatex,sn-mathphys-num]{sn-jnl}

\usepackage{graphicx}%
\usepackage{multirow}%
\usepackage{amsmath,amssymb,amsfonts}%
\usepackage{amsthm}%
\usepackage{mathrsfs}%
\usepackage[title]{appendix}%
\usepackage{xcolor}%
\usepackage{textcomp}%
\usepackage{manyfoot}%
\usepackage{booktabs}%
\usepackage{algorithm}%
\usepackage{algorithmicx}%
\usepackage{algpseudocode}%
\usepackage{listings}%
\usepackage[style=numeric]{biblatex}
\newtheorem{theorem}{Theorem}[section]
\newtheorem{definition}{Definition}[section]
\newtheorem{lemma}[theorem]{Lemma}
\newtheorem{corollary}[theorem]{Corollary}
\newtheorem{Proposition}[theorem]{Proposition}
\newtheorem{remark}[theorem]{Remark}

\begin{document}

\title[Dynamical stability of various convex graphical translators]{Dynamical stability of various convex graphical translators}


\author*[1]{\fnm{Junyoung} \sur{Park}}\email{jp2453@math.rutgers.edu}

\affil*[1]{\orgdiv{Department of Mathematics}, \orgname{Rutgers University}, \orgaddress{\street{110 Frelinghuysen
Road}, \city{Piscataway}, \postcode{08854-8019}, \state{NJ}, \country{USA}}}


\abstract{In the first part of the paper, we prove the existence of longtime solution to mean curvature flow starting from a graph of a continuous function defined over a slab. Then, we establish dynamical stability results for various types of graphical translators to mean curvature flow, namely the grim reaper, two dimensional graphical translators, and asymptotically cylindrical translators. }




\maketitle
\tableofcontents
\section{\centering {Introduction}}
A smooth one parameter family of hypersurfaces $\{M_t\}_{t \in I}$ evolves by mean curvature flow if the normal velocity is equal to the mean curvature vector, i.e
\begin{equation}
    \langle \frac{\partial X}{\partial t}, \nu \rangle = \langle \Vec{H}, \nu \rangle,
\end{equation}
where $\Vec{H}$ is the mean curvature vector, and $\nu$ is the unit normal vector. Since the seminal work of Brakke \cite{brakke}, mean curvature flow has been extensively studied during the past several decades. One way to understand how mean curvature flow deforms hypersurfaces is to look at some special type of hypersurfaces which reduces the analysis. For example, one can consider graphical hypersurfaces over a given domain $\Omega \subset \mathbf{R}^n$. \\

Ecker and Huisken \cite{eckerhuisken}, \cite{Ecker1991InteriorEF} showed by maximum principle arguments that when the initial hypersurface $M_0$ is an entire graph of a locally Lipchitz function, then there exists a longtime smooth solution to mean curvature flow starting from $M_0$ which its time slices are also entire graphs. Furthermore, they showed that if $M_0$ is `straight at infinity', then after suitable rescaling, the hypersurfaces converge to a self expander, a special type of hypersurfaces which move by expansion under mean curvature flow. Later, Clutterbuck \cite{Clutterbuck2005ParabolicEW} (see also \cite{bowlstability}) relaxed local Lipchitz continuity to just continuity. Uniqueness of entire graph solutions starting from locally Lipchitz, proper graph was proved by Daskalopolous and S\'aez \cite{Daskalopoulos2021UniquenessOE}. \\

In the case when $\Omega \subset \mathbf{R}^n$ is a general domain, S\'aez and Schn\" urer \cite{graphicalmcfoverdomain} showed the existence of solutions to mean curvature flow when the initial data is a graph of a locally Lipchitz, proper function over $\Omega$. They further showed that such solutions remain graphical over evolving domain. \\

Our first result is the longtime existence of a graphical solution to mean curvature flow starting from a graph of a continuous function defined over $\Omega^n_b = \mathbf{R}^{n-1}\times (-b,b)$ for some $b > 0$, analogous to results in \cite{eckerhuisken}, \cite{Clutterbuck2005ParabolicEW}.  
\begin{theorem}\label{existence of graphical solution rough statement intro}
   For each $b > 0$, $n \geq 1$, let $\Omega^n_b = \mathbf{R}^{n-1} \times (-b,b)$. For given $u_0 \in C^0(\Omega^n_b)$, there exists $u \in C^{\infty}(\Omega^n_b \times (0, \infty)) \cap C^0(\Omega^n_b \times [0, \infty))$ so that $u$ solves the initial value problem
    \begin{equation}
        \begin{cases}
         \frac{\partial u}{\partial t} = \sqrt{1 + |Du|^2}\textit{\emph{div}}(\frac{Du}{\sqrt{1 + |Du|^2}}) \text{ in }\Omega^n_b \times (0,\infty) \\
         u(x, 0) = u_0 \text{ in }\Omega^n_b
        \end{cases}
    \end{equation}
    If we further assume that for each $J > 0$, one can find $\delta > 0$ so that $|u_0(x)| > J$ for all $x \in \Omega^n_b$ with $b - \delta < |x_n| < b$, then we can find a solution $u$ so that each time slice also has the same property. In particular, the graphs of $u(\cdot, t)$ are complete, graphical hypersurfaces in $\mathbf{R}^{n+1}$ moving by mean curvature.
\end{theorem}
\begin{remark}
    Compared to results in \cite{graphicalmcfoverdomain}, although we are only working with $\mathbf{R}^{n-1}\times (-b,b)$ instead of general domains, our result relaxes the local Lipchitz continuity to just continuity. We are also assuming a slightly weaker condition on $u_0$ than it being a proper function in view of the case when the graph of $u_0$ is a grim reaper plane.
\end{remark}

If we move away from the graphical case,  the flow in general develops singularity in finite time, and cannot be continued within the smooth category. \\

To understand what goes wrong, one has to analyze how singularities are formed. A typical approach is via a blowup analysis. As one approaches singular time, one `magnifies' the flow near the singularity to obtain an approximate picture which models the singularity formation. These approximate pictures are called blowup limits. \\ 

It turns out that by choosing the blowup sequence in a particular way, these limit flows are self similar. For example, Huisken \cite{Huisken1990AsymptoticbehaviorFS} showed that type I singularity formation can be modeled on self shrinking solutions. On the other hand, by following Hamilton's blowup procedure in \cite{hamilton1995formation}, Huisken and Sinestrari \cite{huisken1999mean}, \cite{huisken1999convexity} showed that type II singularity formation of mean convex mean curvature flow can be modeled on translating solutions.\\

In this paper, we will be interested in graphical translators, namely hypersurfaces $M \subset \mathbf{R}^{n+1}$ that are graphs of some function over some domain $\Omega \subset \mathbf{R}^n$, and satisfies
\begin{equation}\label{translator equation}
    \langle \Vec{H}, \nu\rangle  = \langle \nu, e_{n+1}\rangle 
\end{equation}
where $\Vec{H}$ is the mean curvature vector, $\nu$ is the unit normal vector, and $e_{n+1} = (0,0,..,1) \in \mathbf{R}^{n+1}$ is the standard $n+1^{th}$ unit vector in $\mathbf{R}^{n+1}$. These are precisely the hypersurfaces which generate translating solutions to mean curvature flow by the formula 
\begin{equation}
    \mathbf{R}\ni t \to M + te_{n+1}.
\end{equation}
 Thus, in view of previous discussions, translators model type II singularity formation, and hence have been one of the central objects in the study of mean curvature flow. \\

In this paper, we focus on the dynamical stability of various graphical translators. Stability of solitons is a widely studied subject, and tell us how `rigid' the soliton is under perturbations. Stability is usually studied in two different contexts, linear stability and dynamical stability. In the setting of translators, Ilmanen \cite{ilmanen1994elliptic} showed that translators are minimal hypersurfaces in Euclidean space under the conformal metric
\begin{equation}
    g_{ij} = e^{\frac{2}{n}x_{n+1}}\delta_{ij}.
\end{equation}
Therefore, linear stability of translators refers to the non-negativity of the Jacobi operator
\begin{equation}
    Lf = -(\Delta f + \langle e_{n+1}, \nabla f\rangle  + |A|^2).
\end{equation}
In \cite{Xin2015-ei}, Xin showed that all graphical translators are area minimizing, hence linearly stable.\\

Dynamical stability addresses the following question.
\begin{itemize}
    \item Given a graphical translator, will a `small perturbation' of that translator converge back to the initial translator under mean curvature flow?
\end{itemize}
In this paper, we prove dynamical stability results for the grim reaper, two dimensional graphical translators, and asymptotically cylindrical translators.\\

We first look at the one dimensional mean curvature flow, also called curve shortening flow. Mullins \cite{mullins1956two} introduced a graphical translator in $\mathbf{R}^2$, which is now called the grim reaper. It turns out that this is the unique one dimensional translator up to rigid motion (see \cite{Halldorsson2010SelfsimilarST}). Wang and Wo \cite{WANG_WO_2011} proved a dynamical stability result for grim reaper.\\

We generalize the result of Wang and Wo \cite{WANG_WO_2011} by weakening the assumptions on the perturbation.
\begin{theorem}\label{dynamical stability of grim reaper intro}
Let $\overline{M} = \text{Graph}(\overline{u})$ be the grim reaper. If $u_0 : (-\frac{\pi}{2}, \frac{\pi}{2}) \to \mathbf{R}$ satisfies
\begin{equation}
   \|u_0(x) - \overline{u}(x)\|_{C^0(-\frac{\pi}{2}, \frac{\pi}{2})} < \infty,
\end{equation}
and the graph of $u_0$ has bounded curvature and finite total curvature, then there exists a longtime solution $u$ to graphical curve shortening flow starting from $u_0$, so that
\begin{equation}
    \lim_{t \to \infty}u(x,t) - t = \overline{u}(x) + c_0 \text{ in }C^{\infty}_{loc}(-\frac{\pi}{2}, \frac{\pi}{2}),
\end{equation}
where 
\begin{equation}
    c_0 = \frac{1}{\pi}\int_{-\frac{\pi}{2}}^{\frac{\pi}{2}}u_0(x) - \overline{u}(x)dx.
\end{equation}
\end{theorem}
\begin{remark}\label{usual c2 closeness grim reaper}
Wang and Wo \cite{WANG_WO_2011} assumed that $\|u_0 - \overline{u}\|_{C^{2 + \alpha}(-\frac{\pi}{2}, \frac{\pi}{2})} < \infty$, and $M_0$ has finitely many inflection points. On the other hand, in theorem \ref{dynamical stability of grim reaper intro}, we do not impose any form of convexity assumption. Also, one can easily see that finite curvature condition follows from 
    \begin{equation}
        \|u_0 - \overline{u}\|_{C^{2}(-\frac{\pi}{2}, \frac{\pi}{2})} < \infty.
    \end{equation}
    Thus our assumptions are indeed weaker than the $C^{2 + \alpha}$ closeness condition in \cite{WANG_WO_2011}.
\end{remark}
\begin{remark}\label{no improvement in grim reaper case}
    Theorem \ref{dynamical stability of grim reaper intro} implies that one cannot obtain an analogous result to theorem 1.1 in \cite{bowlstability} for grim reapers. Even if $u_0 \in C^{\infty}(-\frac{\pi}{2}, \frac{\pi}{2})$ is \textbf{identical} to $\overline{u}$ for $x \notin (-\epsilon, \epsilon)$ for any small $\epsilon > 0$, if $u_0 > \overline{u}$ in $(-\epsilon, \epsilon)$, then $c_0 > 0$ in theorem \ref{dynamical stability of grim reaper intro}, hence we do not have
    convergence to $\overline{M} + te_{n+1} = \text{Graph}(\overline{u} + t)$, but instead to its vertical translate. 
\end{remark} 

We now consider surfaces in $\mathbf{R}^3$. Altschuler and Wu \cite{altschuler1994translating} showed the existence of bowl soliton in $\mathbf{R}^3$, which is a rotationally symmetric, graphical translator. Hoffman et.al \cite{Hoffman2018GraphicalTF} constructed a one parameter family of strictly convex, graphical translators defined over a slab. In the same paper, by building on previous work of Spruck and Xiao \cite{spruck2020complete}, the authors also showed that in $\mathbf{R}^3$, the only graphical translators are the (tilted) grim reaper plane, 2d bowl soliton, and the one parameter family of translators they constructed, which are called $\Delta$-wings.\\ 

We prove dynamical stability of two dimensional graphical translators which are defined over a slab. The case of bowl soliton is dealt with in theorem \ref{dynamical stability of cylindrical translators intro}.  
\begin{theorem}\label{dynamical stability for 2d translators intro}
    Let $\overline{M} = \text{Graph}(\overline{u})$ be a graphical translator defined over $\Omega^2_b = \mathbf{R}\times (-b, b)$, and not contained in $\Omega^2_{b'}\times \mathbf{R}$ for any $b' < b$. Let $u_0 \in C^0(\Omega^2_b)$ be a convex function in the sense that the graph of $u_0$ encloses a convex region in $\mathbf{R}^3$, and
\begin{equation}
   \|u_0 - \overline{u}\|_{C^0(\Omega_b^2)}\leq C_0 < \infty.
\end{equation}
Then, there exists a longtime solution $u$ to the graphical mean curvature flow starting from $u_0$, so that for each $t_j \to \infty$, by possibly passing through subsequence, we have
$$\lim_{j \to + \infty}u(x_1,x_2,t + {t}_j) -t_j  = \overline{u}(x_1 + c_1, x_2) + t+c_0 \text{ in }C^{\infty}_{loc}(\Omega^2_b \times \mathbf{R})$$
for some $c_0, c_1 \in \mathbf{R}$. In other words, the perturbed flow converges to the translating solution $\overline{M} + te_3$ modulo translation in $x_1x_3$-direction. If $\overline{M}$ is a (tilted) grim reaper plane, then $c_1 = 0$, $c_0 \in [-C_0, C_0]$. If $\overline{M}$ is a $\Delta$-wing (hence $b > \frac{\pi}{2})$, then $|c_1|\tan\theta + |c_0| \leq C_0$ where $\theta = \arccos\frac{\pi}{2b}$. 
\end{theorem}
\begin{remark}
    In view of theorem 2.14 in \cite{dynamicsconvex}, there is a unique convex mean curvature flow starting from graph of $u_0$. In particular, theorem \ref{dynamical stability for 2d translators intro} implies that the convex mean curvature flow starting from graph of $u_0$ has to be graphical over $\Omega^2_b$, and converge locally smoothly to $\overline{M}$ modulo space translation. On the other hand, all solutions we are considering are noncompact, thus it is unclear if convexity is preserved for all possible evolution. 
\end{remark}
\begin{remark}
    Unlike theorem \ref{dynamical stability of grim reaper intro}, we have to allow the possibility of an extra translation in $x_1$-variable in addition to the vertical ($x_3$-direction) translation. This is because an $x_1$-direction translate of a $\Delta$-wing can be bounded by two vertical translates of the same $\Delta$-wing (see remark \ref{extra x1 translation} for details). The extra $x_1$-direction translation is superfluous for (tilted) grim reaper planes. 
\end{remark}

We now consider higher dimensional translators. Clutterbuck et.al \cite{bowlstability} proved the existence, uniqueness, and dynamical stability result for the bowl soliton in all dimensions. Later, Haslhofer \cite{Haslhofer2014UniquenessOT} proved a uniqueness result of the bowl soliton among uniformly two convex, non-collapsed translators. \\

In general, there is a whole family of graphical translators. Hoffman et.al \cite{Hoffman2018GraphicalTF} constructed $n-1$-parameter family of graphical translators in $\mathbf{R}^{n+1}$ which are not in general rotationally symmetric. In $\mathbf{R}^4$, Choi, Haslhofer, and Hershcovits \cite{Choi2021ClassificationON} obtained a classification of non-collapsed translators, which are either the 3d bowl soliton, 2d bowl soliton$\times \mathbf{R}$, or a one parameter family of graphical translators constructed in \cite{Hoffman2018GraphicalTF}, which the authors call HIMW class. Very recently, Bamler and Lai \cite{Bamler2025-thclassification}, \cite{Bamler2025-pdeodi} fully classified all asymptotically cylindrical flows, in particular all asymptotically cylindrical translators, which turns out to be the bowl soliton and non-collapsed entire graph translators constructed in \cite{Hoffman2018GraphicalTF} with possible extra Euclidean factors. In particular, they showed that the asymptotically cylindrical translators are precisely the non-collapsed translators in any dimensions.\\

By closely following the terminologies in \cite{Bamler2025-pdeodi}, \cite{Bamler2025-thclassification}, we first define the notion of asymptotically $(n,k)$-cylindrical translators.
\begin{definition}\label{def of asymptotically cylindrical translators.}
Let $n \geq 2$, $1 \leq k \leq n-1$. A translator $M$ is asymptotically $(n,k)$- cylindrical if for any $\lambda_j \to 0$, the parabolically rescaled sequence of solutions 
\begin{equation}
   M^{j}_t = \lambda_j (M + \lambda_j^{-2}te_{n+1}) \ \ t\in (-\infty, 0) 
\end{equation}
converges locally smoothly to a multiplicity one shrinking cylinder solution
\begin{equation}
    M^{n,k}_{cyl} = \{\mathbf{S}^{n-k}(\sqrt{-2(n-k)t}) \times \mathbf{R}^k\}_{t\in (-\infty, 0)}
\end{equation}
as $j \to \infty$.
\end{definition}
The following theorem concerns the dynamical stability of asymptotically $(n,k)$-cylindrical translators. 
\begin{theorem}\label{dynamical stability of cylindrical translators intro}
Let $n \geq 2$, and $1 \leq k \leq n-1$. Let $\overline{M} = \text{Graph}(\overline{u})$ be an asymptotically $(n,k)$-cylindrical translator (see definition \ref{def of asymptotically cylindrical translators.}). Let $u_0 : \mathbf{R}^n \to \mathbf{R}$ so that
\begin{equation}
    (i)\  \|u_0 - \overline{u}\|_{C^0(\mathbf{R}^n)} \leq C_0 < \infty, \ \ \ (ii)\  M_0 = \text{Graph}(u_0) \text{ is mean convex}.
\end{equation}
Then there exists a longtime solution $u$ to graphical mean curvature flow starting from $u_0$, so that for each $t_j \to \infty$, by possibly passing through subsequence, we have
\begin{equation}
    \lim_{j \to \infty}u(x, t +t_j) - t_j = \overline{u}(x + p_1)+t + p_2 \text{ in }C^{\infty}_{loc}(\mathbf{R}^n\times \mathbf{R}),
\end{equation}
where $(p_1, p_2) \in \mathbf{R}^n \times \mathbf{R} = \mathbf{R}^{n+1}$. In other words, the perturbed solution converges locally smoothly to the translating solution $\overline{M} + te_{n+1}$ modulo space translation. The amount of translation $(p_1, p_2)$ has norm comparable to $C_0$.
\end{theorem}
\begin{remark}
   The modulo space translation statement is necessary due to examples such as $u_0 = \overline{u} + c$ for some $c \neq 0$.
\end{remark}
\begin{remark}
One can see from the proof of theorem \ref{dynamical stability of cylindrical translators intro} that the mean convexity condition in theorem \ref{dynamical stability of cylindrical translators intro} is only used to exclude potential high multiplicity tangent flows. In particular, one can impose different assumption in place of (ii), and still get the same result. In particular, we have the following corollary of the proof of theorem \ref{dynamical stability of cylindrical translators intro}.    
\end{remark}
\begin{corollary}\label{alternative dynamical stability}
    Following the notations in theorem \ref{dynamical stability of cylindrical translators intro}, suppose
    \begin{equation}\label{alternative assumption on initial data cylindrical flows}
        (i) \ \|u_0 - \overline{u}\|_{C^0(\mathbf{R}^n)} \leq C_0 < \infty, \ \ \ (ii) \ \lambda(M_0) < 2\lambda(\overline{M}),
    \end{equation}
    where $\lambda$ denotes the Colding-Minicozzi entropy (see definition \ref{colding minicozzi definition}). Then the same conclusion as theorem \ref{dynamical stability of cylindrical translators intro} holds.
\end{corollary}
\begin{remark}
    Comparing with the dynamical stability result of Clutterbuck et.al \cite{bowlstability}, our theorem considers all possible non-collapsed translators. Moreover, we still obtain a stability result without a strong assumption on the asymptotic behavior of the initial graph at farfield. On the other hand, we require mean convexity, which can be viewed as a $C^2$ closeness condition. Indeed, by the translator equation \eqref{translator equation} written using $\overline{u}$, we have
    \begin{equation}
        Q(D\overline{u}, D^2\overline{u}) = (\delta_{ij} - \frac{D_i\overline{u}D_j\overline{u}}{1 + |D\overline{u}|^2})D^2_{ij}\overline{u} = 1.
    \end{equation}
    Since $Q$ is smooth with respect to its arguments, we see that even after perturbing $D\overline{u}, D^2\overline{u}$ by a small amount, $Q$ will still be positive, meaning the perturbed surface is still mean convex.
\end{remark}
In the case the initially given translator is a bowl soliton with a Euclidean factor, by additionally assuming a similar asymptotic behavior at farfield as in \cite{bowlstability}, we can remove the modulo space translation statement in theorem \ref{dynamical stability of cylindrical translators intro} by using the barriers constructed in \cite{bowlstability}.
\begin{corollary}\label{strong convergence to bowl with strong asymptotic behavior}
Following the notations in theorem \ref{dynamical stability of cylindrical translators intro}, let $\overline{M} = \text{Graph}(\overline{u})$ be a $n-k+1$-dimensional bowl soliton times $\mathbf{R}^{k-1}$. In other words, if we write $x = ( \hat{x},\Tilde{x}) \in \mathbf{R}^{n-k+1}\times \mathbf{R}^{k-1}$, then \begin{equation}
    \overline{u}(x) = \overline{u}(\hat{x})
\end{equation}    
is the graph function of a $n-k+1$-dimensional bowl soliton in $\mathbf{R}^{n-k+2}$. If we assume that
\begin{equation}\label{strongasymptoticbehaviorassumption}
   \|u_0 - \hat{u}\|_{C^0(\mathbf{R}^n)} < \infty, \ \lim_{|\hat{x}| \to \infty}\sup_{\Tilde{x}\in \mathbf{R}^{k-1}}|u_0(x) - \overline{u}(x)| = 0, 
\end{equation}
and $M_0 = \text{Graph}(u_0)$ is mean convex, then we have 
\begin{equation}
    \lim_{t \to +\infty}u(x,t) - t = \overline{u}(x) \text{ in }C^{\infty}_{loc}(\mathbf{R}^n).
\end{equation}
\end{corollary}
\begin{remark}
     We note that although corollary \ref{strong convergence to bowl with strong asymptotic behavior} is already proved in \cite{bowlstability} for bowl solitons $(k = 1)$, the result does not immediately follow from the arguments in \cite{bowlstability} in the case the initially given translator has an extra Euclidean factor. This is because the proof in \cite{bowlstability} heavily relies on the fact that the set where one cannot control $u - \overline{u}$ by the barriers constructed in \cite{bowlstability} is precompact. This property is no longer true when one adds an extra Euclidean factor. 
\end{remark}
\subsection*{\centering{Outline of the proof strategy}}\label{generalstrategy} We explain the main strategy to proving dynamical stability of various graphical translators. Classical regularity theory for graphical solutions to mean curvature flow established in \cite{eckerhuisken}, \cite{Ecker1991InteriorEF}, and theorem \ref{existence of graphical solution rough statement intro} implies the existence of longtime graphical solutions to mean curvature flow starting from the perturbed hypersurface. Let us denote the solution by $\{M_t\}_{t \in [0, \infty)}$. We then look at the forward limit of the flow. For any $t_j \to \infty$, we define
\begin{equation}
    M^j_t = M_{t + t_j} - t_je_{n+1}, \ t\in [-t_j, \infty),
\end{equation}
where $e_{n+1} = (0,0,..,1) \in \mathbf{R}^{n+1}$ is the velocity of the initially given translator. Note that in all the theorems we shall prove, we are assuming $C^0$ closeness. This means that one can use two copies of translating solutions generated by the given translator as upper and lower barriers. This will give us a uniform $C^0$ estimate for above sequence of solutions independently of $j$, which can be upgraded to uniform $C^k$ estimate. Thus, one can take a smooth subsequential limit. The limit flow is also a graphical solution to mean curvature flow, which is eternal and lies between two translating solutions.\\

The first task is to show that the limit flow is also a graphical translating solution with the same velocity $e_{n+1}$. For theorem \ref{dynamical stability of grim reaper intro}, and theorem \ref{dynamical stability for 2d translators intro}, this will be achieved by using the Harnack inequality of Hamilton \cite{Hamilton1995HarnackEF}. For theorem \ref{dynamical stability of cylindrical translators intro}, this will be done by showing that the limit flow we obtained is an asymptotically cylindrical flow. Then by using the recent  classification result of \cite{Bamler2025-thclassification}, together with the fact that the limit flow is eternal, bounded between two translating solutions, we can conclude that the limit flow is a non-collapsed translator of the same velocity as the two barrier solutions. We then use existing classification theorems of graphical translators to conclude that the limit flow is actually equal to the translating solutions generated by the initial translator up to space translation. \\

We emphasize that our strategy heavily relies on classification theorems for translators. The strategy itself can actually be applied to more general convex graphical translators, and will tell us that in the limit, one gets a convex graphical translator that can be bounded by two copies of the initial translator. However, to show that the limit translator is equal to the one we started with, one needs a strong understanding of how various translators look like. To the author's knowledge, currently available classification results of translators are
\begin{itemize}
    \item Uniqueness of one dimensional translators (see \cite{Halldorsson2010SelfsimilarST}).
     \item Uniqueness of bowl solitons (see \cite{Haslhofer2014UniquenessOT}).
    \item Classification of graphical translators in $\mathbf{R}^3$ (see \cite{Hoffman2018GraphicalTF}).
   
    \item Classification of asymptotically cylindrical translators (see \cite{Bamler2025-thclassification}, \cite{Bamler2025-pdeodi}, \cite{Choi2021ClassificationON}).
\end{itemize}
In particular, classification of higher dimensional graphical translators which are defined over higher dimensional slabs remains open, hence our method cannot yet be used to cover this case as well. \\

We give an outline of this paper. In section \ref{preliminaries section}, we set up frequently used notations, and recall several concepts that will be useful throughout the paper. In section \ref{existence of graphical solution section}, we prove theorem \ref{existence of graphical solution rough statement intro}. In section \ref{dynamical stability of grim reaper section}, we prove theorem \ref{dynamical stability of grim reaper intro}. In section \ref{dynamical stability of graphical translators in R3 section}, we prove theorem \ref{dynamical stability for 2d translators intro}. Lastly, in section \ref{asymptoticallycylindricalsection}, we prove theorem \ref{dynamical stability of cylindrical translators intro}.
\section{\centering{Notations and preliminaries}}\label{preliminaries section}

In this section, we set up notations that we will be frequently using, and collect preliminaries that will be relevant to our work. \\

We will be working with graphical solutions to mean curvature flow, either over the whole $\mathbf{R}^n$ or a slab. 
\begin{definition}[n-dimensional slab]
    For each $n \geq 1$, $b > 0$, we define the $n$-dimensional slab of width $b > 0$ to be
    \begin{equation}
        \Omega^n_b = \mathbf{R}^{n-1}\times (-b,b) \subset \mathbf{R}^{n}.
    \end{equation}
\end{definition}
We next define subsets in $\Omega^n_b$ which will play the role of balls in $\mathbf{R}^n$. 
\begin{definition}\label{Some domains and quantities}
    For each $r > 0$, $\lambda \in (0, b)$, we define
    \begin{equation}\label{definition of domain}
    K_{r, \lambda} = B_{\mathbf{R}^{n-1}}(0, r) \times (-b + \lambda, b - \lambda) \subset \Omega^n_b.
    \end{equation}
\end{definition}
\begin{definition}
For each $1 \leq i \leq n+1$, we let \begin{equation}
    e_i = (0,..,1,0..0) \in \mathbf{R}^{n+1}
\end{equation}
to be the standard $i^{th }$basis vector of $\mathbf{R}^{n+1}$.    
\end{definition}

Now, we recall some basic properties of graphical mean curvature flow. It is well known that a family of hypersurfaces which are graphs of function, i.e
\begin{equation}
    M_t = \text{Graph}(u(\cdot, t))
\end{equation}
evolves by mean curvature flow if and only if $u$ solves the graphical mean curvature flow equation
\begin{equation}
    \frac{\partial u}{\partial t} = \sqrt{1 + |Du|^2}\text{div}(\frac{Du}{\sqrt{1 + |Du|^2}}).
\end{equation}
We follow the notations in \cite{eckerhuisken}.
\begin{definition}[Commonly used quantities]\label{def of commonly used quantities}
For graphical solution to mean curvature flow
\begin{equation}
    \{M_t =\text{Graph}(u(\cdot, t))\}_{t \in I},
\end{equation}
we define
\begin{align*}
    &\nu = \frac{1}{\sqrt{1 + |Du|^2}}(-Du, 1) = \text{upward unit normal vector field},\\
    &v = \frac{1}{\langle \nu, e_{n+1}\rangle } = \sqrt{1 + |Du|^2}, \\
    &A_{ij} = \frac{D^2_{ij}u}{\sqrt{1 + |Du|^2}} = \text{second fundamental form}, \\ & H = \textit{\emph{div}}(\frac{Du}{\sqrt{1 +|Du|^2}}) = \text{mean curvature},
\end{align*}
where $Du$ is the spatial derivative of $u(\cdot, t)$. 
\end{definition}
We now recall the `ancient pancake' solution to mean curvature flow constructed in theorem 1.1 in \cite{ancientpancake}. We note that the $x_n$-axis will play the role of $x_1$-axis in the original statement of theorem 1.1 in \cite{ancientpancake}.
\begin{definition}[Ancient pancake solution]\label{def of ancient pancake}
For each $n \geq 1$, there exists a compact, convex, $O(1)\times O(n)$ symmetric ancient solution to mean curvature flow which lies in the slab
$\Omega = \{x \in \mathbf{R}^{n+1}\ 
 |\ |x_n| < \frac{\pi}{2}\}$, denoted by $\{\Sigma^n_t\}_{t \in (-\infty, 0)}$. Furthermore, as $t \to -\infty$, we have the following `radius' estimate
\begin{align}\label{model pancake asymptotic data}
   & \min_{p \in \Sigma_t}|p| = |P(e_n, t)| \geq \frac{\pi}{2} - o(\frac{1}{(-t)^k}) \text{ for any }k > 0, \\ &\max_{p \in \Sigma_t}|p| = |P(\phi, t)| = -t + (n-1)\ln(-t) + C + o(1) \text{ for any unit vector $\phi \in \{e_n\}^{\perp}$.} 
\end{align}
\end{definition}
Next, we recall Hamilton's differential Harnack inequality for mean curvature flow \cite{Hamilton1995HarnackEF}. 
\begin{theorem}\label{Hamilton's Harnack inequality}[Differential Harnack inequality]
    Let $\{M_t\}_{t \in (-\alpha, T)}$ be a convex, smooth solution to mean curvature flow which is either compact, or complete with bounded curvature. Then for any tangent vector $V$, the following inequality holds. 
    \begin{equation}
        Z(V) = A(V,V) + 2\nabla_VH + \frac{H}{2(t+\alpha)} + \nabla_tH \geq 0.
    \end{equation}
    Here, $A$ is the second fundamental form, $H$ is the mean curvature. In particular, if the flow is ancient i.e $\alpha = \infty$, then for any tangent vector $V$, we have
    \begin{equation}\label{ancient harnack inequality}
        Z(V) = A(V,V) + 2\nabla_VH + \nabla_tH \geq 0.
    \end{equation}
\end{theorem}
Next, we recall the definition of Colding-Minicozzi entropy, which was first introduced in \cite{coldingminicozzi}. We state the definition for Radon measures in $\mathbf{R}^{n+1}$. One can then understand the entropy of a hypersurface $M$ as the entropy of the Radon measure $\mathcal{H}^n|_{M}$, where $\mathcal{H}^n$ is the $n$-dimensional Hausdorff measure.
\begin{definition}\label{colding minicozzi definition}
    Let $\mu$ be a Radon measure in $\mathbf{R}^{n+1}$ which satisfies 
    \begin{equation}
        \sup_{r > 0}\frac{\mu(B(0, r))}{r^n} < \infty.
    \end{equation}
    Then for each $x_0 \in \mathbf{R}^{n+1}$, $t > 0$, define
    \begin{equation}
        F_{x_0, t}(\mu) = \int_{\mathbf{R}^{n+1}}\frac{1}{(4\pi t)^{n/2}}\exp{(-\frac{|x - x_0|^2}{4t})}d\mu.
    \end{equation}
    Then, the entropy of $\mu$ is defined as
    \begin{equation}
        \lambda(\mu) = \sup_{x_0 \in \mathbf{R}^{n+1}, t > 0}F_{x_0, t}(\mu).
    \end{equation}
\end{definition}
Two key facts about entropy is that it is invariant under translation / dilation, and it is non-increasing under mean curvature flow. \\

Lastly, we recall the recent classification of all asymptotically cylindrical mean curvature flow by Bamler-Lai \cite{Bamler2025-pdeodi}, \cite{Bamler2025-thclassification}. Here, we only state the relevant parts of their work that we will use in section \ref{asymptoticallycylindricalsection}. We refer the reader to the original papers for the complete statement.\\

We first define the space of all `normalized' asymptotically $(n,k)$-cylindrical translators (see definition \ref{def of asymptotically cylindrical translators.}). 
\begin{definition}\label{Definition of MCFsoliton}
$\textup{\textbf{MCF}}^{n,k}_{\textup{soliton}}$ is the space of asymptotically $(n,k)$-cylindrical translating solution, whose $0$-time slice contains the origin, and is invariant under reflection across some collection of $n$-pairwise orthogonal hyperplanes, which pass through the origin and may depend on the flow.
\end{definition}
We now recall the key quantities that parametrize $\textup{\textbf{MCF}}^{n,k}_{\textup{soliton}}$. To do so, we first recall theorem 6.1, proposition 7.1, lemma 7.6 in \cite{Bamler2025-pdeodi}. We summarize the relevant parts in the following proposition. 
\begin{Proposition}\label{asymptotic bevaior of nkflows}
    Let $\mathcal{M} = \{M_t\}_{t \in (-\infty, 0)}$ be an asymptotically $(n,k)$-cylindrical translating solution that is not a bowl soliton with possible Euclidean factors. Then define its type I rescaling
    \begin{equation}
        \Tilde{M}_{\tau} = e^{\frac{\tau}{2}}M_{-e^{-\tau}}, \ \tau (t) = -\ln (-t).
    \end{equation}
Then for sufficiently negative $\Tilde{\tau} > -\infty$, for all $\tau \in (-\infty, \Tilde{\tau}]$, define
\begin{equation}
    R(\tau) = 10\sqrt{\ln(\Tilde{\tau} - \tau + 10)}, \ \ \mathcal{D}_{\tau} = \mathbf{S}^{n-k}(\sqrt{2(n-k)}) \times  B(0, R(\tau)).
\end{equation}
Then there exists a smooth function $u_{\tau} \in C^{\infty}(\mathcal{D}_{\tau})$ so that
\begin{equation}
    \{((1 + u_{\tau})y, x )\ |\ (y,x) \in \mathcal{D}_{\tau}\} \subset \Tilde{M}_{\tau} .
\end{equation}
Moreover, there exists 
\begin{equation}
    U^+(\tau) = U_1 + U_{\frac{1}{2}} +U_0 + .. + U_{-10} : (-\infty, \Tilde{\tau}) \to \mathcal{V}_{1}\oplus\mathcal{V}_{\frac{1}{2}}\oplus\mathcal{V}_{0}\oplus..\oplus \mathcal{V}_{-10},
\end{equation}
and a smooth function
\begin{equation}
    V(\tau) = \sum_{1 \leq i,j\leq k}v_{ij}(\tau)p_{ij}(x) : (-\infty, \overline{T}) \to \mathcal{V}_0,
\end{equation}
where $\mathcal{V}_{\lambda}$ is the $\lambda$-eigenspace of the linearized rescaled mean curvature flow equation at $M^{n,k}_{-1} = \mathbf{S}^{n-k}(\sqrt{2(n-k)}) \times \mathbf{R}^k$ in the usual Gaussian weighted $L^2$ space, and $p_{ij}(x) = \frac{1}{2\sqrt{2}}(x_ix_j - 2\delta_{ij})$ where $(y,x) \in M^{n,k}_{-1}$. These smooth functions satisfy the following properties.
\begin{equation}\label{asymptotic expansion existence by bamler lai}
    (i)\ \|u_{\tau} - U^+(\tau)\|_{C^{0}(\mathcal{D}_{\tau})} \leq \frac{1}{10}(\Tilde{\tau} - \tau + 10)^{-11},
\end{equation}
\begin{equation}\label{appx asymptotic expansion by bamler lai}
   (ii)\ \|U_0(\tau) - V(\tau)\|_{L^2_{w}} \leq C(\Tilde{\tau} - \tau + 10)^{-3} 
\end{equation}
for all sufficiently small $\tau$, where $\|\cdot\|_{L^2_{w}}$ denotes the Gaussian weighted $L^2$ norm. \\\\
(iii) If we view $V(\tau)$ as a $k\times k$ matrix $\{v_{ij}(\tau)\}$, then $V(\tau) \in \mathbf{R}^{k\times k}_{\leq 0}$, $\lim_{\tau \to -\infty}V(\tau) = 0$, and solves the equation
\begin{equation}\label{the ode solved by appx asymptotic expansion}
   \partial_{\tau}V = -\sqrt{2}V^2 + 2\textup{tr}(V^2)V + C^*(n-k)V^3
\end{equation}
with $C^*$ being some dimensional constant. This $V$ uniquely determines a nonnegative $k\times k$ matrix, denoted by $Q(\mathcal{M})$. 
\end{Proposition}
We now recall the definition of $Q$ and $b$, and state the main classification result in \cite{Bamler2025-thclassification}.
\begin{definition}\label{definition of parameter Q,b in bamler lai}[Definition 7.7 in \cite{Bamler2025-pdeodi}, definition 1.5 in \cite{Bamler2025-thclassification}]
For each $\mathcal{M} \in \textup{\textbf{MCF}}^{n,k}_{\textup{soliton}}$, if $\mathcal{M}$ is a bowl soliton with a possible Euclidean factor, then $Q(\mathcal{M}) = 0$. Otherwise, we define $Q(\mathcal{M})$ as in proposition \ref{asymptotic bevaior of nkflows}. Also, we define
    \begin{equation}
        b(\mathcal{M}) = \frac{v}{|v|^2}.
    \end{equation}
    where $v$ is the velocity of the translating solution $\mathcal{M}$. 
\end{definition}
\begin{theorem}\label{parametrization of space of solitons by bamler lai}[Theorem 1.3, theorem 1.6 in \cite{Bamler2025-thclassification}]
Any eternal, asymptotically $(n,k)$-cylindrical mean curvature flow is a translating solution, which after suitable space translation belongs to $\textup{\textbf{MCF}}^{n,k}_{\textup{soliton}}$. Moreover, let $\mathcal{M}, \mathcal{M}' \in \textup{\textbf{MCF}}^{n,k}_{\textup{soliton}}$. If
\begin{equation}
    b(\mathcal{M}) = b(\mathcal{M}'), \ Q(\mathcal{M}) = Q(\mathcal{M}'),
\end{equation}
then $\mathcal{M} = \mathcal{M}'$.
\end{theorem}
Lastly, we record a useful property of solutions to the ODE \eqref{the ode solved by appx asymptotic expansion}.
\begin{lemma}\label{key property of ODE solution}[Lemma 7.32 in \cite{Bamler2025-pdeodi}]
    Let $V_1, V_2$ be two ancient solutions to the ODE \eqref{the ode solved by appx asymptotic expansion}. Suppose $V_1, V_2 \to 0$ as $\tau \to -\infty$, and $\|V_1(\tau) - V_2(\tau)\|\leq C|\tau|^{-3}$ for sufficiently negative $\tau$, then $V_1 = V_2$.
\end{lemma}
\section{\centering{Existence of graphical solution}}\label{existence of graphical solution section}
In this section, we prove the existence of longtime graphical solution to mean curvature flow over a slab $\Omega^n_b = \mathbf{R}^{n-1} \times (-b, b)$ with continuous initial data. Following the ideas in \cite{bowlstability}, we construct a sequence of entire graph solutions to mean curvature flow with approximate initial data. We show that one can take a $C^{\infty}_{loc}(\Omega^n_b)$ subsequential limit, which is the desired graphical solution over a slab. Due to standard quasilinear PDE theory, this can be done by obtaining an apriori $C^1$ estimate of graphical solutions over $\Omega^n_b$. $C^0$ estimate is obtained by using ancient pancakes of Bourni et.al \cite{ancientpancake} as barriers. Then, one can apply the gradient estimate of Colding - Minicozzi \cite{gradientestimate} to obtain the desired $C^1$ estimate. \\

We restate theorem \ref{existence of graphical solution rough statement intro} for the reader's convenience. 
\begin{theorem}[Theorem \ref{existence of graphical solution rough statement intro}]\label{Graphical solution existence}
    For each $b > 0$, $n \geq 1$, let $\Omega^n_b = \mathbf{R}^{n-1} \times (-b,b)$. For given $u_0 \in C^0(\Omega^n_b)$, there exists $u \in C^{\infty}(\Omega^n_b \times (0, \infty)) \cap C^0(\Omega^n_b \times [0, \infty))$ so that $u$ solves the initial value problem
    \begin{equation}\label{Grahical mcf IVP}
        \begin{cases}
         \frac{\partial u}{\partial t} = \sqrt{1 + |Du|^2}\textit{\emph{div}}(\frac{Du}{\sqrt{1 + |Du|^2}}) \text{ in }\Omega^n_b \times (0,\infty) \\
         u(x, 0) = u_0 \text{ in }\Omega^n_b
        \end{cases}
    \end{equation}
    If we further assume that for each $J > 0$, one can find $\delta > 0$ so that $|u_0(x)| > J$ for all $x \in \Omega^n_b$ with $b - \delta < |x_n| < b$, then we can find $u$ so that each time slice also has the same property. In particular, the graphs of $u(\cdot, t)$ are complete, graphical hypersurfaces in $\mathbf{R}^{n+1}$ moving by mean curvature.
\end{theorem}
We now begin the proof of theorem \ref{Graphical solution existence}. We first establish a $C^0$ estimate. 
\begin{lemma}[$C^0$ estimate]\label{C^0 estimate}
    Let $u \in C^{\infty}(\Omega^n_b \times (0, \infty)) \cap C^0(\Omega^n_b \times [0, \infty))$ be the solution to the initial value problem \eqref{Grahical mcf IVP}. Then, for each $T > 0$, $r > 0$, $\lambda \in (0,\frac{1}{2}\min(\frac{\pi}{2},b))$, there exists $R = R(n, \lambda, T)$, and $Q = Q(n,\lambda, T)$ so that we have the $C^0$ estimate
 \begin{equation}\label{C0 estimate equation}
     \|u\|_{C^0(K_{r, 2\lambda}\times [0, T])} \leq \|u_0\|_{C^0(K_{ r + R, \lambda})} + Q(n, \lambda, T),
 \end{equation}
 where $K_{ r, \lambda}$ is given by \eqref{definition of domain}.  Moreover, $R$ is given by
\begin{align}\label{radius equation}
    R(n, \lambda,  T) = \frac{\pi (2T + T_0(n))}{2\lambda} + \frac{2(n-1)\lambda}{\pi}\ln (\frac{\pi^2(2T + T_0(n))}{4\lambda^2})+ \frac{2\lambda C(n)}{\pi} + 1,
\end{align}
and $Q$ is given by
\begin{equation}\label{extra c0 quantity estimate}
    Q(n, \lambda, T) = \frac{\pi T}{2\lambda} + \frac{2(n-1)\lambda}{\pi}\ln 2 + 1
\end{equation}
\end{lemma}
\begin{proof}[Proof of lemma \ref{C^0 estimate}]
Let $\{\Sigma^n_t\}_{t\in(-\infty, 0)}$ be the ancient pancake solution given in definition \ref{def of ancient pancake}. Then in view of \eqref{model pancake asymptotic data}, one can find fixed $T_0 = T_0(n) > 0$ so that whenever $t \leq -T_0$, then
\begin{align}\label{model pancake radius estimate}
     0 < \frac{\pi}{2}- \min_{p \in \Sigma_t}|p| < \frac{1}{4} , \ \ \ |\max_{p \in \Sigma_t}|p| - (-t + (n-1)\ln(-t) + C)| < \frac{1}{4}.
\end{align}
Then, for each $\lambda \in (0,\frac{1}{2}\min(\frac{\pi}{2}, b))$, we can consider the parabolically rescaled ancient pancakes
\begin{equation}\label{definition of rescaled ancient pancakes}
    \Sigma^{n, \lambda}_t = \frac{2\lambda}{\pi} \Sigma^n_{(\frac{2\lambda}{\pi})^{-2}t}.
\end{equation}
Set 
\begin{equation}\label{def of long and short radius}
    r_s(t,\lambda) = \min_{p \in \Sigma^{n, \lambda}_t}|p|, \ \ \ r_l(t, \lambda) =  \max_{p \in \Sigma^{n, \lambda}_t}|p|.
\end{equation}
The inequalities \eqref{model pancake radius estimate} with $\lambda \leq \frac{\pi}{2}$ imply that whenever $t \leq -T_0(n)$, then $(\frac{2\lambda}{\pi})^{-2}t \leq t \leq -T_0(n)$, hence
\begin{equation}\label{short radius estimate for rescaled ancient pancakes}
    0 < \lambda - r_s(t,\lambda)   < \frac{\lambda}{2\pi} \leq \frac{1}{4},
\end{equation}
and
\begin{equation}\label{long radius estimate for rescaled ancient pancakes}
    |r_l(t, \lambda) - (-\frac{\pi t}{2\lambda}+ \frac{2(n-1)\lambda}{\pi}\ln(-\frac{\pi^2t}{4\lambda^2}) + \frac{2\lambda C}{\pi})| \leq \frac{\lambda}{2\pi} \leq \frac{1}{4}.
\end{equation}
We now prove lemma \ref{C^0 estimate}. Let $u$ be given as in lemma \ref{C^0 estimate}. Fix any $x_0 \in K_{ r, 2\lambda}$, and $T > 0$. We consider the $-(T_0 + 2T)$-time slice of translated rescaled ancient pancakes
\begin{equation}\label{def of upper barriers}
    M^{n, \lambda, x_0, +}_{-T_0 - 2T } = \Sigma^{n, \lambda}_{-T_0 -2T}  + x_0 + (\|u_0\|_{C^0(K_{ r + R, \lambda})} + r_l(-T_0-2T, \lambda) + \frac{1}{2})e_{n+1},
\end{equation}
and
\begin{equation}\label{def of lower barriers}
    M^{n, \lambda, x_0, -}_{-T_0 - 2T } = \Sigma^{n, \lambda}_{-T_0 -2T}  + x_0 - (\|u_0\|_{C^0(K_{r + R, \lambda})} + r_l(-T_0-2T, \lambda) + \frac{1}{2})e_{n+1},
\end{equation}
where $r_l$ is given by \eqref{def of long and short radius}. 
Then by definitions \eqref{def of upper barriers}, \eqref{def of lower barriers}, inequalities \eqref{short radius estimate for rescaled ancient pancakes}, \eqref{long radius estimate for rescaled ancient pancakes}, and the definition of $R$ in \eqref{radius equation}, we see that
\begin{equation}\label{good boundary behavior}
    M^{n, \lambda, x_0,+}_{-T_0 - 2T},M^{n, \lambda, x_0,-}_{-T_0 - 2T} \subset K_{r + R, \lambda} \times \mathbf{R},
\end{equation}
and 
\begin{equation}\label{initial avoidance}
     M^{n, \lambda, x_0,+}_{-T_0 - 2T},M^{n, \lambda, x_0,-}_{-T_0 - 2T}, \text{ and graph of }u_0 \text{ are pairwise disjoint}.
\end{equation}
We now evolve $M^{n, \lambda, x_0,+}_{-T_0 - 2T},M^{n, \lambda, x_0,-}_{-T_0 - 2T}, \text{ and graph of }u_0$ by mean curvature flow up to time $t = T$. Due to \eqref{good boundary behavior} and \eqref{initial avoidance}, we can conclude by avoidance principle that  
\begin{equation}\label{later avoidance}
     M^{n, \lambda, x_0,+}_{-T_0 - 2T + t},M^{n, \lambda, x_0,-}_{-T_0 - 2T+ t}, \text{ and graph of }u(\cdot, t) \text{ are pairwise disjoint for all }0 \leq t \leq T.
\end{equation}
Note that by \eqref{def of upper barriers}, \eqref{def of lower barriers} and the symmetry of ancient pancake solution, the centers of $M^{n, \lambda, x_0,+}_{-T_0 - 2T + t},M^{n, \lambda, x_0,-}_{-T_0 - 2T+ t}$ lie on the line $\{x_0\} \times \mathbf{R}$. Then by \eqref{later avoidance}, we can obtain the estimate
\begin{equation}\label{C0 bound}
    |u(x_0, t)| \leq \|u_0\|_{C^0(K_{ r + R, \lambda})} + \frac{1}{2} + r_l(-T_0 - 2T, \lambda) - r_l(-T_0 -T, \lambda)
\end{equation}
Because $-T_0 - 2T, -T_0 - T \leq -T_0$, we can apply \eqref{long radius estimate for rescaled ancient pancakes}. This gives us 
\begin{align*}
        \frac{1}{2} + r_l(-T_0 - 2T, \lambda) &- r_l(-T_0 -T, \lambda)  \\ \leq &\frac{1}{2} + [(\frac{\pi (T_0 + 2T)}{2\lambda}+ \frac{2(n-1)\lambda}{\pi}\ln(\frac{\pi^2(T_0 + 2T)}{4\lambda^2}) + \frac{2\lambda C}{\pi}) + \frac{1}{4}]\\ & - [(\frac{\pi (T_0 + T)}{2\lambda}+ \frac{2(n-1)\lambda}{\pi}\ln(\frac{\pi^2(T_0 + T)}{4\lambda^2}) + \frac{2\lambda C}{\pi}) - \frac{1}{4}]\\ \leq & \frac{\pi T}{2\lambda} + \frac{2(n-1)\lambda}{\pi}\ln 2 + 1
\end{align*}
Combining with \eqref{C0 bound}, and recalling the definition of $Q$ given by \eqref{extra c0 quantity estimate}, we finally have the desired estimate \eqref{C0 estimate equation}.
\end{proof}
Once we have a $C^0$ bound, we can apply the gradient estimate in \cite{gradientestimate} to obtain a $C^1$ estimate. Let us first recall the gradient estimate in \cite{gradientestimate}.
\begin{lemma}[Theorem 1 in \cite{gradientestimate}]\label{gradient estimate of colding minicozzi}
   There exists $C = C(n)$ so that if the graphs of $u(\cdot, t) : B(0, \mu)  \to \mathbf{R}$, $t \in  [0, \frac{\mu^2}{1 + 2n}]$ flows by mean curvature, then
   \begin{equation}
       |Du|(0, \frac{\mu^2}{4n(1 + 2n)}) \leq \exp(C(1 + \frac{\|u(\cdot, 0)\|_{\infty}}{\mu})^2).
   \end{equation}
\end{lemma}
\begin{lemma}[Gradient estimate]\label{gradient estimate}
  Let $u \in C^{\infty}(\Omega^n_b \times (0, \infty)) \cap C^0(\Omega^n_b \times [0, \infty))$ be the solution to the initial value problem \eqref{Grahical mcf IVP}. Then, for each $0 < T_1 < T_2$, $r > 0$, $\lambda \in (0,\frac{1}{3}\min(\frac{\pi}{2},b))$, we have the $C^1$ estimate   
  \begin{equation}\label{gradient estimate final}
    \|Du\|_{C^0(K_{r, 3\lambda} \times [T_1, T_2])} \leq \exp{(C(n)(1 + \frac{\|u\|_{C^0(K_{r + \lambda, 2\lambda}\times [0, T_2])}}{\min(\lambda, \sqrt{T_1})})^2}). 
  \end{equation}
  In particular, by combining with lemma \ref{C^0 estimate}, we have
   \begin{equation}\label{gradient estimate only initial data}
    \|Du\|_{C^0(K_{r, 3\lambda} \times [T_1, T_2])} \leq \exp{(C(n)(1 + \frac{\|u_0\|_{C^0(K_{r + \lambda + R, \lambda})} + Q}{\min(\lambda, \sqrt{T_1})})^2}),
  \end{equation}
  where $R = R(n,\lambda, T_2)$, and $Q = Q(n, \lambda, T_2)$ are given by \eqref{radius equation} and \eqref{extra c0 quantity estimate} respectively.
\end{lemma}
\begin{proof}[Proof of lemma \ref{gradient estimate}]
Define 
\begin{equation}\label{choice of mu}
        \mu = \min(\lambda, \sqrt{4n(1 + 2n)T_1}) > 0.
    \end{equation}
    For each $x_0 \in K_{r, 3\lambda}$,
    one has
    \begin{equation}
         B_{\mathbf{R}^n}(x_0, \mu) \subset K_{r + \lambda, 2\lambda}.
    \end{equation} Also, for each $T_1 \leq t_0 \leq T_2$, we can find $\overline{t} \in [0, T_2]$ so that $\overline{t} + \frac{\mu^2}{4n(1 + 2n)} = t_0$.  
   We apply lemma \ref{gradient estimate of colding minicozzi} to $u$ restricted onto $B(x_0, \mu) \times [\overline{t}, \overline{t} + \frac{\mu^2}{1 + 2n}]$. Then we have
    \begin{equation}\label{grad estimate first applied}
        |Du|(x_0, t_0) \leq \exp(C(1 + \frac{\|u(\cdot, \overline{t})\|_{C^0(B(x_0, \mu))}}{\mu})^2).
    \end{equation}
    By the fact that $ B_{\mathbf{R}^n}(x_0, \mu) \subset K_{r + \lambda, 2\lambda}$, and $\overline{t} \in [0, T_2]$, we trivially have
    \begin{equation}
        \|u(\cdot, \overline{t})\|_{C^0(B(x_0, \mu))} \leq \|u\|_{C^0(K_{r + \lambda, 2\lambda \times [0, T_2])}}.
    \end{equation}
    Combining above with \eqref{grad estimate first applied}, using the definition \eqref{choice of mu}, and taking the supremum among all $(x_0, t_0) \in K_{r, 3\lambda} \times [T_1, T_2]$, we obtain
    \begin{equation}
        \|Du\|_{C^0(K_{r, 3\lambda} \times [T_1, T_2])} \leq \exp{(C(n)(1 + \frac{\|u\|_{C^0(K_{r + \lambda, 2\lambda}\times [0, T_2])}}{\min(\lambda, \sqrt{T_1})})^2}).
    \end{equation}
We can estimate the $C^0$ norm on the right hand side above by lemma \ref{C^0 estimate}. This gives us the final estimate.
\end{proof}
Once we have an interior gradient estimate, we can apply classical interior curvature estimates of Ecker and Huisken in \cite{eckerhuisken} to obtain higher order derivative estimates to graphical solution $u$ by restricting it to small standard parabolic cylinders as we did in the proof of lemma \ref{gradient estimate}. This implies the following result. 
\begin{lemma}\label{higher order derivative estimate}
    Let $u \in C^{\infty}(\Omega^n_b \times (0, \infty)) \cap C^0(\Omega^n_b \times [0, \infty))$ be the solution to the initial value problem \eqref{Grahical mcf IVP}. Then, for each $0 < T_1 < T_2$, $r > 0$, $\lambda \in (0,\frac{1}{4}\min(\frac{\pi}{2},b))$, $k \geq 0$, we can find constant $C_k = C_k(n, \lambda, T_1, T_2, \|u\|_{C^0(K_{r+ 2\lambda + R(n, \lambda, T_2), \lambda})}) > 0$ so that
    \begin{equation}
        \|D^ku\|_{C^0(K_{r, 4\lambda} \times [T_1, T_2])} \leq C_k.
    \end{equation}
\end{lemma}
Now that we have all the key estimates, we can prove theorem \ref{Graphical solution existence}. 
\begin{proof}[Proof of theorem \ref{Graphical solution existence}]
We define a sequence of smooth cutoff functions $\{\phi_j\}_{j \in \mathbf{N}} \subset C^{\infty}_c(\mathbf{R}^n)$ which is identically equal to 1 in $K_{j, \frac{1}{j}}$, and vanishes outside $K_{j+1, \frac{1}{j+1}}$, where $K_{r, \lambda}$ are the domains given by \eqref{definition of domain}. We also choose some $\eta \in C^{\infty}_c(B(0, \frac{1}{10}))$ with $\eta > 0$, and $\int_{B(0, \frac{1}{10})}\eta dx = 1$, and construct standard mollifier
\begin{equation}
    \eta_j(x) = \frac{1}{j^n}\eta(jx).
\end{equation}
We now define sequence of approximate initial data by
\begin{equation}\label{approximate initial data}
    u^j_0(x) = \eta_j * (u_0\phi_j) \in C_c^{\infty}(\mathbf{R}^n).
\end{equation}
Then, by classical results of Ecker and Huisken in \cite{eckerhuisken}, there exist sequence of functions $u^j \in C^{\infty}(\mathbf{R}^n \times [0, \infty))$ which solves the initial value problem
\begin{equation}
    \begin{cases}
        \frac{\partial u^j}{\partial t} = \sqrt{1 + |Du^j|^2}\text{div}(\frac{Du^j}{\sqrt{1 + |Du^j|^2}}) \text{ in }\mathbf{R}^n \times [0, \infty) \\
        u^j(x, 0) = u^j_0(x) \text{ in }\mathbf{R}^n.
    \end{cases}
\end{equation}
We then restrict $u^j$ onto $\Omega^n_b \times [0, \infty)$. By the definition of the initial data given by \eqref{approximate initial data}, for each $r>0$ and $\lambda \in (0, \frac{1}{100}\min(\frac{\pi}{2}, b))$, we have
\begin{equation}\label{key c0 property for appx entire solution}
    \|u^j_0\|_{C^0(K_{r, \lambda})} \leq \|u_0\|_{C^0(K_{r, \lambda})}
\end{equation}
for all sufficiently large $j$ depending on $r, \lambda$. Then by applying lemma \ref{higher order derivative estimate}, we see that for each $r > 0$, $\lambda \in (0, \frac{1}{100}\min(\frac{\pi}{2}, b))$, $0 < T_1 < T_2$, we can find a constant $C$ which is independent of $j$ so that 
\begin{equation}\label{uniform estimate}
    \|D^ku^j\|_{C^0(K_{r, \lambda} \times [T_1, T_2])} \leq C < \infty.
\end{equation}
for all sufficiently large $j$. Therefore, by applying standard Arzela Ascoli theorem and diagonal argument, we can find $u \in C^{\infty}(\Omega^n_b \times (0, \infty))$ so that
\begin{equation}
    u^j \to u \text{ in }C^{\infty}_{loc}(\Omega^n_{b}\times (0, \infty)).
\end{equation}
Since each $u^j$ solves the graphical mean curvature flow, $u$ also solves the graphical mean curvature flow.\\

We now show that $u$ is continuous up to $t = 0$, in other words 
\begin{equation}
    \lim_{t \to 0}u(x,t) = u_0(x) \text{ for each }x \in \Omega_b.
\end{equation}
This can be established by comparison argument with spheres. By continuity of $u_0$, for each $x_0 \in \Omega^n_b$, $\epsilon > 0$, one can choose $0 < \delta < \epsilon$ so that 
\begin{equation}
    u_0(x_0) - \epsilon < u_0(y) < u_0(x_0) + \epsilon
\end{equation}
for all $|x_0 - y| < \delta$. 
Moreover, since $u^j_0 \to u_0$ locally uniformly in $C^0$ norm, we can further assume that for all $j$ large, 
\begin{equation}
     u_0(x_0) - 2\epsilon < u_0^j(y) < u_0(x_0) + 2\epsilon
\end{equation}
for all $|x_0 - y| < \delta$. Then, we can find two spheres $S^+(x_0, \epsilon, \delta)$, and $S^-(x_0, \epsilon, \delta)$ so that both spheres have radius $\delta$, and the centers of the spheres are $x_0 + (u_0(x_0) + 2\epsilon + \delta)e_{n+1}$, and $x_0 + (u_0(x_0) - 2\epsilon - \delta)e_{n+1}$ respectively. Then, we see that $S^+(x_0, \epsilon, \delta)$, $S^-(x_0, \epsilon, \delta)$, and the graph of $u^j_0$ are pairwise disjoint. Then, by avoidance principle, the evolution of $S^+(x_0, \epsilon, \delta)$, and $S^-(x_0, \epsilon, \delta)$ by mean curvature flow acts as a barrier for the graph of $u^j$. Combining this with the fact that $\delta < \epsilon$, we see that for all $0 \leq t \leq \frac{\delta^2}{4n}$, 
\begin{equation}
    u_0(x_0) - 3\epsilon < u^j(x_0, t) < u_0(x_0) + 3\epsilon
\end{equation}
for all large $j$. Thus, by letting $j \to + \infty$, we can conclude that for each $\epsilon > 0$, $x_0\in \Omega^n_b$, we can find $\delta > 0$ so that for all $0 \leq t \leq \frac{\delta^2}{4n}$, we have
\begin{equation}
    |u(x_0, t) - u_0(x_0)| < 3\epsilon,
\end{equation}
which means 
\begin{equation}
    \lim_{t \to 0}u(x, t) = u_0(x) \text{ for all }x \in \Omega_b.
\end{equation}
This completes the proof of the first half of theorem \ref{Graphical solution existence}.\\

We now prove the second half of theorem \ref{Graphical solution existence}. We assume that 
\begin{equation}\label{uniformly blowing up near boundary}
    \forall J > 0, \exists\delta > 0 \ \text{so that }|u_0(x_1,..,x_n)| > J \ \forall x \in \Omega^n_b \text{ which }b - |x_n| < \delta.
\end{equation}
For simplicity, we assume that $u_0(x) \to \infty$ as $x \to \partial \Omega^n_b$. Other possible cases can be dealt with by essentially the same argument. We first note that the arguments in the proof of first half of theorem \ref{Graphical solution existence} hold for any approximate sequence of smooth (up to $t = 0$) entire solution as long as we have \eqref{key c0 property for appx entire solution} for large $j$. Let 
\begin{equation}
    \overline{M} = \text{Graph}(\overline{u})
\end{equation}
with
\begin{equation}
    \overline{u} : \mathbf{R}^{n-1}\times (-\frac{\pi}{2}, \frac{\pi}{2}) \to \mathbf{R}, \ \ \overline{u}(x_1,..,x_n) = \ln \cos x_{n}.
\end{equation}
Note that $\overline{M}_t = \overline{M} - te_{n+1}$ moves by mean curvature flow. We define the approximate initial data as follows. As in the first part of the proof, we choose some $\eta \in C^{\infty}_c(B(0, \frac{1}{10}))$ with $\eta > 0$, and $\int_{B(0, \frac{1}{10})}\eta dx = 1$, and construct standard mollifier
\begin{equation}
    \eta_j(x) = \frac{1}{j^n}\eta(jx).
\end{equation}
Now, for each $j \geq 1$, we define
\begin{equation}
    u^j(x) = \min \{u_0(x), j\} * \eta_j(x), \ x\in \mathbf{R}^n.
\end{equation}
Note that by assumption \eqref{uniformly blowing up near boundary}, $u^j \in C^{\infty}(\mathbf{R}^n)$ is well defined. Moreover, due to \eqref{uniformly blowing up near boundary}, we see that the approximate initial data $\{u^j\}_{j \geq 1}$ indeed have uniform estimate \eqref{key c0 property for appx entire solution}. Thus, by previous argument, we get a longtime graphical solution to mean curvature flow starting from $u_0$ by taking the limit of $u^j$ as $j \to \infty$. We will show that the limit flow obtained by this choice of approximate initial data has the desired behavior near the boundary. Fix any $J,T > 0$. We can find $0 < \delta = \delta_{J,T} <1$ so that for all $j$ sufficiently large, 
\begin{equation}
    u^j(x_1,..,x_n) > 2J + T \text{ for all }|x_n| > b - \delta.
\end{equation}
We will show that $u^j(\cdot, t)$ has the desired behavior near the boundary for all $t \in [0, T]$, $j$ sufficiently large. We consider the following translates of $\overline{M}$, namely
\begin{equation}
    M_0 = \overline{M} + (b + \frac{\pi}{2} - \delta)e_{n} + (-\ln\cos(\frac{\pi}{2} - \delta) + 2J + T)e_{n+1},
\end{equation}
and
\begin{equation}
    M_1 = \overline{M} - (b + \frac{\pi}{2} - \delta)e_{n} + (-\ln\cos(\frac{\pi}{2} - \delta) +  2J + T)e_{n+1},
\end{equation}
Then, one can see that for all sufficiently large $j$, the graph of $u^j$ lies above $M_0$, and $M_1$. Therefore, by avoidance principle, we see that $t$-time slice of the entire graph solution starting from $u^j$ remains above $M_0 - te_{n+1}$, and $M_1 - te_{n+1}$. Then by the choice of $M_0$, $M_1$, we have
\begin{equation}\label{lower bound near boundary}
    u^j(x,t) \geq \ln \cos (b+ \frac{\pi}{2} - \delta -|x_n|)-\ln\cos(\frac{\pi}{2} - \delta) + 2J + T - t.
\end{equation}
for all $ b - |x_n| < \delta$, $t \in [0, T]$. Note that when $|x_n| = b$, the right hand side in \eqref{lower bound near boundary} is equal to $2J + T - t \geq 2J$. Therefore, we can find $\Tilde{\delta} > 0$ so that the right hand side in inequality \eqref{lower bound near boundary} is bounded from below by $J$ whenever $b - |x_n| < \Tilde{\delta}$. Thus, we have
\begin{equation}
    u^j(x_1,..,x_n,t)\geq J \text{ for all }b - |x_n| < \Tilde{\delta}, t \in [0,T].
\end{equation}
Note that $\Tilde{\delta}$ only depends on $u_0$, $J$ and $T$, hence this property will pass through the limit. Thus, each time slice of the graphical solution to mean curvature flow starting from $u_0$ also has property \eqref{uniformly blowing up near boundary}, thus completing the proof of the second part of theorem \ref{Graphical solution existence}.
\end{proof}
\section{\centering{Dynamical stability of grim reaper}}\label{dynamical stability of grim reaper section}
In this section, let $\overline{M} = \text{Graph}(\overline{u})$ be the grim reaper, i.e
\begin{equation}
    \overline{u} : (-\frac{\pi}{2}, \frac{\pi}{2}) \ni x \to -\ln(\cos x).
\end{equation}
We generalize theorem 1.4 in \cite{WANG_WO_2011} by weakening the initial $C^{2+\alpha}$ closeness (condition 1.10 in \cite{WANG_WO_2011}), and dropping the finitely many inflection points condition (condition 1.11 in \cite{WANG_WO_2011}). We also mention that our approach is different from that in \cite{WANG_WO_2011}, where the authors used the monotonicity of
\begin{equation}
    J(t) = \int_{-\frac{\pi}{2}}^{\frac{\pi}{2}}(u_x - \overline{u}_x)^2dx. 
\end{equation}
Here, we follow the general strategy \ref{generalstrategy}, and show that the forward limit has to be a translating solution by using Hamilton's Harnack inequality \eqref{ancient harnack inequality}, and barrier arguments. \\

We restate theorem \ref{dynamical stability of grim reaper intro} for the reader's convenience. 
\begin{theorem}[Theorem \ref{dynamical stability of grim reaper intro}]\label{dynamical stability of grim reaper}
Let $\overline{M} = \text{Graph}(\overline{u})$, with \begin{equation}
    \overline{u} : (-\frac{\pi}{2}, \frac{\pi}{2}) \ni x \to -\ln\cos x.
\end{equation}
Let $u_0 \in C^{2+ \alpha}(-\frac{\pi}{2}, \frac{\pi}{2})$ so that 
\begin{equation}\label{c0 closeness 1d}
    \sup_{x\in (-\frac{\pi}{2}, \frac{\pi}{2})}|u_0(x) - \overline{u}(x)| \leq C_0 \text{ for some }C_0 > 0,
\end{equation}
and $M_0 = \text{Graph}(u_0)$ has bounded curvature and finite total curvature, i.e
\begin{equation}\label{Finite total curvature}
    \sup_{M_0}|\kappa| + \int_{M_0}|\kappa|ds  = C_1 < \infty.
\end{equation}
Then there exists $u \in C^{\infty}((-\frac{\pi}{2}, \frac{\pi}{2})\times (0, \infty))\cap C^2((-\frac{\pi}{2}, \frac{\pi}{2})\times [0, \infty))$ so that $u$ solves the initial value problem
\begin{equation}
    \begin{cases}
        u_t = \frac{u_{xx}}{1 + u_x^2} \text{ in }(-\frac{\pi}{2},\frac{\pi}{2}) \times (0, \infty) \\ 
        u(x,0) = u_0(x) \text{ in }(-\frac{\pi}{2}, \frac{\pi}{2}).
    \end{cases}
\end{equation}
Furthermore,
\begin{equation}
    \lim_{t \to \infty}u(x,t) - t = \overline{u}(x) + c_0 \text{ in }C^{\infty}_{loc}(-\frac{\pi}{2}, \frac{\pi}{2}),
\end{equation}
where 
\begin{equation}
    c_0 = \frac{1}{\pi}\int_{-\frac{\pi}{2}}^{\frac{\pi}{2}}u_0(x) - \overline{u}(x)dx.
\end{equation}
\end{theorem}

\begin{proof}[Proof of theorem \ref{dynamical stability of grim reaper}]
    By theorem \ref{Graphical solution existence}, we can find $u \in C^{\infty}((-\frac{\pi}{2}, \frac{\pi}{2})\times (0, \infty))\cap C^0((-\frac{\pi}{2}, \frac{\pi}{2})\times [0, \infty))$ so that $u$ solves the initial value problem
\begin{equation}
    \begin{cases}
        u_t = \frac{u_{xx}}{1 + u_x^2} \text{ in }(-\frac{\pi}{2},\frac{\pi}{2}) \times (0, \infty) \\ 
        u(x,0) = u_0(x) \text{ in }(-\frac{\pi}{2}, \frac{\pi}{2}),
    \end{cases}
\end{equation}
and since $u_0 \geq \overline{u}(x) - C_0$, we also have
\begin{equation}\label{nice behavior at infinity 1d}
    \lim_{|x| \to \frac{\pi}{2}}u(x, t) = \infty \text{ for all }t\geq 0.
\end{equation}
In fact, since we are assuming $u_0 \in C^{2+ \alpha}$, we can slightly modify the proof of lemma \ref{gradient estimate} and lemma \ref{higher order derivative estimate} and obtain uniform local in space $C^{2+\alpha}$ estimates up to $t = 0$. This will give us $u \in C^2((-\frac{\pi}{2}, \frac{\pi}{2})\times [0, \infty))$. Let 
\begin{equation}
    M_t = \text{Graph}(u(\cdot, t)).
\end{equation}
Then, we see that $\{M_t\}_{t \in [0, \infty)}$ is a complete graphical solution to mean curvature flow with $\text{Graph}(u_0)$ as its initial data.\\

We now consider the total curvature of $M_t$. Define
\begin{equation}\label{def of I}
    I(t) = \int_{M_t}|\kappa|ds = \int_{-\frac{\pi}{2}}^{\frac{\pi}{2}}\frac{|u_{xx}|}{1 + u_x^2}dx,
\end{equation}
where $\kappa$ is the curvature of $M_t$. We claim that $I(t)$ is nonincreasing in $t$. We first note that for each $T > 0$, one can find a constant $C = C(T, C_0, C_1) > 0$ so that
\begin{equation}\label{curvature bound up to any T}
    \sup_{t \in [0, T]}\sup_{M_t}|\kappa| \leq C.
\end{equation}
By assumption \eqref{c0 closeness 1d}, \eqref{Finite total curvature}, and the fact that $M_0$ is graphical solution over $(-\frac{\pi}{2}, \frac{\pi}{2})$, we see that $M_0$ is uniformly $C^1$ close to vertical lines $\{x = \frac{\pi}{2}\} (\{x = -\frac{\pi}{2}))$ as $x \to \frac{\pi}{2} (x \to -\frac{\pi}{2})$. In other words, for any $\epsilon > 0$, we can find $\delta_0 > 0$ so that for any $\frac{\pi}{2} - \delta_0 < |x| < \frac{\pi}{2}$, we can write $M_0$ as a $C^1$ graph over vertical line with height and gradient bounded by $\epsilon > 0$. Then by pseudolocality theorem (see theorem 1.5 in \cite{10.4310/jdg/1547607687}), the same property (uniform $C^1$ closeness to vertical line) hold for all $M_t$, $t \in [0, T]$. Then by interior regularity theorem in \cite{Ecker1991InteriorEF} with assumption \eqref{Finite total curvature}, we indeed have uniform bound on curvature up to $t = T$. Once we have bounded curvature and finite initial total curvature, we can apply lemma 6.1 in \cite{zhulectures} to conclude that 
\begin{equation}
    I(t) \text{ is nonincreasing in }t \text{ in }t \in [0, T].
\end{equation}
Since $T > 0$ was arbitrary, we have the desired monotonicity and finiteness of total curvature. \\

We then consider
\begin{equation}
    I_0 = \lim_{t \to \infty}I(t).
\end{equation}
Note that by monotonicity, this quantity is well defined. Now, for any $t_j \to \infty$, we consider
\begin{equation}
    M^j_t = M_{t + t_j} - t_je_2 \text{ for }t \in [-t_j, \infty).
\end{equation}
By assumption \eqref{c0 closeness 1d} and avoidance principle \cite{White2024TheAP}, if we write $M^j_t = \text{Graph}(u^j(\cdot, t))$, then
\begin{equation}\label{limit flow in between grims}
    \overline{u}(x) +t - C_0 \leq u^j(x,t) \leq \overline{u}(x) +t + C_0.
\end{equation}
Hence, the translating grim reapers acts as barriers, and provide uniform local $C^0$ estimates of $u^j$ independent of $j$. Then by lemma \ref{higher order derivative estimate}, we have uniform local higher order derivative estimates of $u^j$ independent of $j$. Thus, we can extract a subsequential limit $M^{\infty}_t = \text{Graph}(u^{\infty}(\cdot, t))$. \\

The limit flow is a complete, eternal, graphical solution to mean curvature with total curvature bounded by $I_0$, and lies in between two translating grim reapers. We claim that $M^{\infty}_t$ is convex. We follow the argument in proposition 5.4 in \cite{zhulectures}. We claim that whenever $\kappa_{\infty}(p,t) = 0$ at some $p \in M^{\infty}_t$, then $\partial_s\kappa_{\infty}(p,t) = 0$ as well. Suppose this is not the case and there exists $p \in M^{\infty}_{t_0}$ so that \begin{equation}
    \kappa_{\infty}(p, t_0)=0,
\end{equation} yet  
\begin{equation}
    0 < \alpha \leq |\partial_s\kappa_{\infty}(p,t_0)| \leq \beta.
\end{equation}
Then, by smoothness of both the limit flow and the convergence $M^j \to M^{\infty}$, we can find $j_0$, $\delta_0 > 0$ and $p_j(t) \in M^j_t$ so that
\begin{equation}\label{properties of appx nodal pts}
 \kappa_j(p_j(t), t) = 0 , \ 0 < \frac{\alpha}{2} \leq |\partial_s\kappa_j(q, t)| \leq 2\beta  \text{ for all }j \geq j_0, |t - t_0| < \delta, d_{j,t}(q, p_j(t)) \leq \delta,
\end{equation}
where $d_{j,t}$ is the intrinsic distance on $M^j_t$. For each $N > 0$, define a cutoff function
\begin{equation}
    \eta_N \in C_c^{\infty}(\mathbf{R}^2)
\end{equation}
so that 
\begin{equation}
|D\eta_N| + |D^2\eta_N| \leq 3, \ \eta_N = 1 \text{ when }|x| \leq N, \text{ and }\eta_N = 0 \text{ when }|x| \geq N+1.
\end{equation}
Following the computations in lemma 6.1 in \cite{zhulectures}, for all $j \geq j_0$, $|t - t_0| < \delta$, $0 < \epsilon<1 $, and $N$ sufficiently large (but independently of the parameters $j, t, \epsilon$), we have
\begin{align*}
 \frac{d}{ds}\int_{M^j_t}\eta_N(\epsilon^2 + \kappa_j^2)^{1/2}ds \leq &\int_{M^j_t}(2|\kappa_j||D\eta_N| + |D^2\eta_N|)(\epsilon^2 + \kappa_j^2)^{1/2}ds \\ &- \int_{p_j(t) -\epsilon}^{p_j(t) + \epsilon}\epsilon^2(\epsilon^2 + \kappa_j^2)^{-3/2}|\partial_s\kappa_j|^2ds \\ \leq & \int_{M^j_t \cap \{N \leq (x^2 + y^2)^{1/2} \leq N+1 \}}C(\kappa_j^2 +1)ds \\ & - \int_{-\epsilon}^{\epsilon}\epsilon^2(\epsilon^2 + 4\beta^2s^2)^{-3/2}\frac{\alpha^2}{4}ds \\
 \leq & \int_{M^j_t \cap \{N \leq (x^2 + y^2)^{1/2} \leq N+1 \}}C(\kappa_j^2 +1)ds - C(\alpha, \beta).
\end{align*}
Integrating from $t = t_0 -\delta$ to $t = t_0 + \delta$ tells us that 
\begin{align*}
     \int_{M^j_{t_0 + \delta}}\eta_N(\epsilon^2 + \kappa_j^2)^{1/2}ds&  - \int_{M^j_{t_0 - \delta}}\eta_N(\epsilon^2 + \kappa_j^2)^{1/2}ds \\ &\leq -2\delta C(\alpha, \beta) + \int_{t_0 - \delta}^{t_0 + \delta}\int_{M^j_t \cap \{N \leq (x^2 + y^2)^{1/2} \leq N+1 \}}C(\kappa_j^2 +1)ds.
\end{align*}
In view of curvature estimate \eqref{curvature bound up to any T}, total curvature bound, and the definition of $M^j_t$, we can let $\epsilon \to 0$, and let $N \to \infty$. Dominated convergence theorem implies that
\begin{equation}
    I(t_j + t_0 + \delta) - I(t_j + t_0 - \delta) \leq -2\delta C_0(\alpha, \beta) < 0
\end{equation}
for all $j \geq j_0$, where $I$ is given by \eqref{def of I}. This is a contradiction for large $j$, because
\begin{equation}
   \lim_{j \to \infty}I(t_j + t_0 + \delta) - I(t_j + t_0 - \delta) = I_0 - I_0 = 0. 
\end{equation}
This means that all inflection points of $M^{\infty}_t$ are degenerate. By using zero set result in \cite{zeroset} (see also proof of theorem 6.1 in \cite{Haslhofer2016LECTURESOC}), this implies that $M^{\infty}_t$ is convex. It is actually strictly convex, or else by strong maximum principle, the limit flow has to be a straight line which is impossible since it lies between two translating grim reapers. \\

So far, we have seen that the limit flow has to be strictly convex. Also, by once again using finite total curvature and the fact that it lies between two grim reapers, we can obtain uniformly $C^1$ closeness of the limit flow to vertical lines $\{ x = \frac{\pi}{2}\} (\{x = -\frac{\pi}{2}\})$ as $|x| \to \frac{\pi}{2}$. Then local regularity theorem \cite{White2005-nl} implies that the limit flow has local in time uniform bound on the curvature. Thus we can use Hamilton's Harnack inequality \eqref{ancient harnack inequality} on the limit flow. As in proposition 6.1 of \cite{brendlechoi}, we define the vector field
\begin{equation}
    V = Hve_2^{tan},
\end{equation}
where $H, v$ are from definition \ref{def of commonly used quantities}. Applying Harnack inequality with above $V$ implies that
\begin{equation}\label{acceleration of velocity in 1d}
    \partial_t^2u^{\infty} \geq 0.
\end{equation}
We now claim that \begin{equation}
    \partial_tu^{\infty} = 1,
\end{equation}
i.e $M^{\infty}_t$ is a translating solution to mean curvature flow with velocity $e_2$. This will immediately imply that \begin{equation}
    u^{\infty}(x,t) = \overline{u}(x) + t + c_0
\end{equation}
for some $c_0 \in \mathbf{R}$, since the grim reaper is the only one dimensional translator, and our limit flow has to be graphical over $(-\frac{\pi}{2}, \frac{\pi}{2})$. If for some $x_0 \in (-\frac{\pi}{2}, \frac{\pi}{2})$, $t_0 \in \mathbf{R}$, we have
\begin{equation}
    \partial_tu^{\infty}(x_0, t_0) = 1 + \delta > 1.
\end{equation}
Then by \eqref{acceleration of velocity in 1d}, we have for all $t \geq t_0$
\begin{equation}
    \partial_tu^{\infty}(x_0, t) \geq 1 + \delta.
\end{equation}
Then, by integrating in $t$, we have for all $t > s \geq t_0$,
\begin{equation}
    u^{\infty}(x_0, t) - u^{\infty}(x_0, s) \geq (1 + \delta)(t - s).
\end{equation}
On the other hand, by \eqref{limit flow in between grims}, we have
\begin{equation}
    u^{\infty}(x_0, t) - u^{\infty}(x_0, s) \leq (t - s) +2C_0, 
\end{equation}
which is a contradiction for large $t - s$. Thus, $\partial_tu^{\infty} \leq 1$ everywhere. The reverse inequality is obtained similarly. Thus we have
\begin{equation}
    \partial_tu^{\infty} = 1,
\end{equation}
which implies that
\begin{equation}
    \lim_{j \to \infty}u(x, t+ t_j) -t_j = \overline{u}(x) + t + c_0 \text{ in }C^{\infty}_{loc}(-\frac{\pi}{2}, \frac{\pi}{2}),
\end{equation}
for some $c_0 \in \mathbf{R}$. \\

By following \cite{WANG_WO_2011}, we now claim that
\begin{equation}
    c_0 = \frac{1}{\pi}\int_{-\frac{\pi}{2}}^{\frac{\pi}{2}}u_0(x) - \overline{u}(x)dx.
\end{equation}
Note that this claim ensures that the constant $c_0$ is independent of choice of subsequence $t_j \to \infty$, hence we actually have full convergence, rather than subsequential convergence. By simply differentiating 
\begin{equation}
    \phi(t) = \frac{1}{\pi}\int_{-\frac{\pi}{2}}^{\frac{\pi}{2}}u(x,t) - t - \overline{u}(x)dx
\end{equation}
in $t$, we have
\begin{equation}
    \frac{d\phi}{dt}(t) = \frac{1}{\pi}\int_{-\frac{\pi}{2}}^{\frac{\pi}{2}}\frac{u_{xx}}{1 + u_x^2} - 1dx = 0
\end{equation}
where we used the fact that the graph of $u$ is arbitrarily $C^1$ close to the vertical lines near the boundary. Then we see that
\begin{equation}
    c_0 = \lim_{j \to \infty} \phi(t_j) = \phi(0) = \frac{1}{\pi}\int_{-\frac{\pi}{2}}^{\frac{\pi}{2}}u_0(x) - \overline{u}(x)dx,
\end{equation}
thus completing the proof of theorem \ref{dynamical stability of grim reaper}.
\end{proof}
\section{\centering{Dynamical stability of graphical translators in $\mathbf{R}^3$}}\label{dynamical stability of graphical translators in R3 section}
In this section, we focus on evolution of surfaces in $\mathbf{R}^3$. As before, for each $b > 0$, $n\geq 1$, we set $\Omega^n_b = \mathbf{R} \times (-b, b)$, and $\overline{M} = \text{Graph}(\overline{u})$ be a graphical translator over $\Omega^2_b$, and not in any smaller slab. In this section we follow our general strategy \ref{generalstrategy} and show that the evolution of a graph of $u_0 : \Omega_b^2 \to \mathbf{R}$ which is convex, and is close to $\overline{M}$ in $C^0$ sense converges locally smoothly to $\overline{M}$ as $t\to \infty$ modulo space translation.\\

We restate theorem \ref{dynamical stability for 2d translators intro} for the reader's convenience. 
\begin{theorem}[Theorem \ref{dynamical stability for 2d translators intro}]\label{dynamical stability of 2d translators}
Let $\overline{M} = \text{Graph}(\overline{u})$ be a graphical translator defined over $\Omega^2_b = \mathbf{R}\times (-b, b)$, and not contained in $\Omega^2_{b'}\times \mathbf{R}$ for any $b' < b$. Let $u_0 \in C^0(\Omega^2_b)$ be a convex function in the sense that the graph of $u_0$ encloses a convex region in $\mathbf{R}^3$. Suppose there exists $C_0 > 0$ so that 
\begin{equation}\label{C0 closeness 2d}
    \overline{u}(x) - C_0 \leq u_0(x) \leq \overline{u}(x) + C_0 \text{ for all }x \in \Omega^2_b
\end{equation}
Then, there exists $u \in C^{\infty}(\Omega^2_b \times (0, \infty)) \cap C^0(\Omega^2_b\times [0, \infty))$ so that $u$ solves the initial value problem
\begin{equation}
\begin{cases}
    \frac{\partial u}{\partial t} = \sqrt{1 + |Du|^2}\textit{\emph{div}}(\frac{Du}{\sqrt{1 + |Du|^2}}) \text{ in }\Omega^2_b \times (0, \infty) \\
    u(x,0) = u_0(x) \text{ in }\Omega^2_b.
\end{cases}
\end{equation}
Furthermore, for each $t_j \to \infty$, there exists a subsequence $\{\Tilde{t}_j\} \subset \{t_j\}$ so that, 
$$\lim_{j \to + \infty}u(x_1,x_2,t + \Tilde{t}_j) -\Tilde{t}_j  = \overline{u}(x_1 + c_1, x_2) + t+c_0 \text{ in }C^{\infty}_{loc}(\Omega^2_b \times \mathbf{R})$$
for some $c_0, c_1 \in \mathbf{R}$, i.e the flow converges to the translating solution $\overline{M} + te_3$ modulo translation in $x_1x_3$-direction. If $\overline{M}$ is a (tilted) grim reaper plane, then $c_1 = 0$, $c_0 \in [-C_0, C_0]$. If $\overline{M}$ is a $\Delta$-wing (hence $b > \frac{\pi}{2})$, then $|c_1|\tan\theta + |c_0| \leq C_0$ where $\theta = \arccos\frac{\pi}{2b}$. 
\end{theorem}
\begin{remark}\label{independence of constants 2d translators remark}
    It is unclear if the constants $c_0,c_1 \in \mathbf{R}$ in above theorem are independent to the choice of (sub)sequence $\{\Tilde{t}_j\} \subset \{t_j\}$ in general. However, in the case one can find some $\{\Tilde{t}_j\} \subset \{t_j\}$ so that 
    \begin{equation}
        \limsup_{j \to \infty}|\Tilde{t}_{j+1} - \Tilde{t}_j| < \infty
    \end{equation}
    then we can show that the constants are independent of the choice of (sub)sequence, hence we get full convergence, i.e
    \begin{equation}
        \lim_{t \to \infty}u(x_1, x_2, t) - t = \overline{u}(x_1 + c_1, x_2) + c_0 \text{ in }C^{\infty}_{loc}(\Omega^2_b).
    \end{equation}
\end{remark}
\begin{remark}\label{extra x1 translation}
Compared to theorem \ref{dynamical stability of grim reaper}, we have to consider an extra $x_1$-direction translation in the case of $\Delta$-wings. This is due to the following simple lemma. 
\begin{lemma}
    Let $\overline{u} : \Omega_b^2 \to \mathbf{R}$ with $b > \frac{\pi}{2}$ so that 
    \begin{equation}
        \overline{M} = \text{Graph}(\overline{u}) = \Delta \text{-wing}.
    \end{equation}
    Then, for any $c_1 \in \mathbf{R}$, 
    \begin{equation}
        \overline{u}(x_1, x_2) - \tan\theta|c_1| \leq \overline{u}(x_1 + c_1,x_2) \leq \overline{u}(x_1, x_2) + \tan\theta|c_1|
    \end{equation}
    with $\theta = \arccos\frac{\pi}{2b}$.
\end{lemma}
\begin{proof}
    By using theorem 1.5 in \cite{spruck2020complete} (or theorem 2.4 in \cite{Hoffman2018GraphicalTF}), we see that for each fixed $|x_2| < b$, 
\begin{equation}
    \lim_{|x_1| \to \infty}|\frac{\partial \overline{u}}{\partial x_1}(x_1, x_2)| = \tan\theta \text{ with }\theta = \arccos\frac{\pi}{2b}.
\end{equation}
Thus convexity of $\overline{u}$ implies that
\begin{equation}
    |\frac{\partial \overline{u}}{\partial x_1}| \leq \tan\theta \text{ everywhere}.
\end{equation}
Then the lemma follows by integration in $x_1$. 
\end{proof}
Above lemma implies that in the statement of theorem \ref{dynamical stability of 2d translators}, we can take
\begin{equation}
    u_0(x_1, x_2) = \overline{u}(x_1 + \frac{C_0}{\tan \theta}, x_2).
\end{equation}
Then since the graph of $u_0$ evolves only by translation in $e_3$ direction, we see that
\begin{equation}
    \lim_{t \to \infty}u(x,t) - t = u_0(x) = \overline{u}(x_1 + \frac{C_0}{\tan\theta}, x_2).
\end{equation}
This simple example implies that we cannot drop the $x_1$-direction translation in the statement of theorem \ref{dynamical stability of 2d translators}.
\end{remark}
We briefly explain the proof of theorem \ref{dynamical stability of 2d translators}. The idea is again to use general strategy \ref{generalstrategy}. To do so, we must ensure that the forward limit flow is convex, and we can apply Hamilton's Harnack inequality. However, because our solutions are noncompact, and the initial surface is not necessarily $C^1$, it is unclear if one indeed has both convexity, and validity of Harnack inequality. To overcome this, we follow the ideas in \cite{Daskalopoulos2021UniquenessOE} and construct the solution as a local smooth limit of convex, compact solutions to mean curvature flow. \\

We first prove two lemmas that will be used to construct the longtime solution to graphical mean curvature flow starting from $u_0$. Let us first prove two lemmas.
\begin{lemma}\label{splitting off of intial data}
 Suppose $\overline{M} = \text{Graph}(\overline{u})$ is a (tilted) grim reaper plane, i.e
\begin{equation}\label{tilted grim reaper}
    u : \Omega^2_b \to \mathbf{R}, \ u(x_1,x_2) = x_1\tan\theta - \frac{1}{\cos^2\theta}\ln(\cos(x_2\cos\theta) ), \ \theta = \arccos\frac{\pi}{2b}.
\end{equation}
Let $u_0$ be as in theorem \ref{dynamical stability of 2d translators}. Then
\begin{equation}\label{splitting off line}
    u_0(x_1,x_2) = u_0(0, x_2) + x_1\tan\theta.
\end{equation}
\end{lemma}
\begin{proof}[Proof of lemma \ref{splitting off of intial data}]
For each $x_2 \in (-b, b)$, we consider
\begin{equation}
    \mathbf{R}\ni x_1 \to \phi_{x_2}(x_1) = \frac{\partial u_0}{\partial x_1}(x_1,x_2).
\end{equation}
Note that by convexity, $\phi_{x_2}$ is well defined for almost every $x_1$, for each $x_2 \in (-b,b)$. Then to prove \eqref{splitting off line}, it is enough to show that 
\begin{equation}
    \phi_{x_2}(x_1) = \tan \theta \text{ whenever it is well defined}.
\end{equation}
If for some $a \in \mathbf{R}$, $x_2 \in (-b,b)$, $\phi_{x_2}(a) = \tan \theta + \delta >  \tan\theta$, then convexity implies that 
\begin{equation}
    \phi_{x_2}(x_1) \geq \tan\theta + \delta \text{ for a.e }x_1 \geq a.
\end{equation}
Then by integrating, we see that
\begin{equation}
    u_0(x_1,x_2) - u_0(a, x_2) \geq (\tan \theta + \delta)(x_1 - a) \text{ for all }x_1 \geq a.
\end{equation}
On the other hand, by \eqref{C0 closeness 2d}, we have
\begin{equation}
    u_0(x_1,x_2) - u_0(a, x_2) \leq (\tan \theta )(x_1 - a) + 2C_0 \text{ for all }x_1 \geq a.
\end{equation}
This is a contradiction when $x_1 - a$ is large. The case when $\phi_{x_2}(a) = \tan \theta - \delta < \tan\theta$ is treated similarly. Thus, we ultimately have 
\begin{equation}
 \phi_{x_2}(x_1) = \tan \theta \text{ whenever it is well defined}. 
\end{equation}
Thus by integration, we have
\begin{equation}
    u_0(x_1,x_2) = u_0(0, x_2) + x_1\tan\theta,
\end{equation}
completing the proof of lemma \ref{splitting off of intial data}.
\end{proof}
We now show how one can construct a convex graphical solution to mean curvature flow on which the Harnack inequality holds, provided the initial data is a graph of a convex, proper function.  
\begin{lemma}\label{existence of convex evolution when deltawing}
Let $n\geq 1$, $b > 0$, and $\Omega^n_b = \mathbf{R}^{n-1}\times(-b,b)$, $u_0 \in C^0(\Omega^n_b)$ so that it is convex. Assume further that $u_0$ is proper, i.e the preimage of any compact set is compact. Then there exists $u \in C^{\infty}(\Omega^n_b \times (0, \infty)) \cap C^0(\Omega^n_b\times [0, \infty))$ so that $u$ solves the initial value problem
\begin{equation}
\begin{cases}
    \frac{\partial u}{\partial t} = \sqrt{1 + |Du|^2}\textit{\emph{div}}(\frac{Du}{\sqrt{1 + |Du|^2}}) \text{ in }\Omega^n_b \times (0, \infty) \\
    u(x,0) = u_0(x) \text{ in }\Omega^n_b.
\end{cases}
\end{equation}
The graphs of $u(\cdot, t)$ are complete, embedded, convex hypersurfaces in $\mathbf{R}^{n+1}$. Moreover, Hamilton's Harnack inequality holds on the graphs of $u(\cdot, t)$.
\end{lemma}
\begin{proof}[Proof of lemma \ref{existence of convex evolution when deltawing}]
The proof is a slight modification of proof of corollary 5.2 in \cite{Daskalopoulos2021UniquenessOE}. By possibly translating $u_0$ vertically, we may assume that 
\begin{equation}\label{normalization}
\min_{x \in \Omega_b}u_0(x) = 0.    
\end{equation}
Now, for each $i \geq 1$, define

\begin{equation}\label{def of auxilary appx domain}
    D^i_0 = \{(x, x_{n+1}) \in \Omega^n_b\times \mathbf{R}\ | \ u_0(x) \leq x_{n+1} \leq 2i - u_0(x), \ u_0(x) \leq i\} \text{ for }i \geq 1.
\end{equation}

Convexity, and \eqref{normalization} implies that $D^i_0 \neq \emptyset$, and are convex. By the assumption that $u_0$ is proper, $D^i_0$ are compact. Let $M^i_0 = \partial D^i_0$. Then $M^i_0$ are compact, convex, Lipchitz continuous, and reflection symmetric with respect to $\{x_3 = i\}$. Then each $M^i_0$ generates compact, convex solutions to mean curvature flow which are reflection symmetric with respect to $\{x_3 = i\}$, denoted by $\{M^i_t = \partial D^i_t\}_{t \in [0, T_i)}$. \\

Let $\hat{M}_t^i = M_t^i \cap \{x_3 \leq i \}$ be the lower half of $M_t^i$. For each $r > 0$, $\lambda \in (0, b)$, we define
\begin{equation}\label{def of domain2}
    K_{r, \lambda} = (-r, r) \times (-b + \lambda, b - \lambda) \subset \Omega_b.
\end{equation}
We claim that for each $T > 0$, $r > 0$, $\lambda > 0$, $\{\hat{M}^i_t\}_{t \in [0, T]}$ is graphical over $K_{r, \lambda}$. Let for each $i\geq 1$, 
\begin{equation}
    p^i = (0, 0, i) \in \mathbf{R}^3.
\end{equation}
Let $\{\Sigma_t\}_{t \in (-\infty, 0)}$ be the ancient pancake introduced in definition \ref{def of ancient pancake}. For each $\lambda \in (0, b)$, we rescale $\Sigma_t$ by factor of $\frac{2(3b - \lambda)}{3\pi}$, i.e we let
\begin{equation}
    \Tilde{\Sigma}_t=\Tilde{\Sigma}^{\lambda}_t = \frac{2(3b - \lambda)}{3\pi}\Sigma_{(\frac{2(3b - \lambda)}{3\pi})^{-2}t}.
\end{equation}
Then, in view of the \eqref{model pancake asymptotic data}, we can choose $\Tilde{T} = \Tilde{T}(n, \lambda. r, T) > 0$ so large that 
\begin{equation}
    K_{2r, \frac{\lambda}{2}} \subset \pi(\Tilde{\Sigma}_t) \text{ for all }-\Tilde{T} \leq t \leq -\Tilde{T} + T,
\end{equation}
where 
\begin{equation}
    \pi : \mathbf{R}^{n+1} \to \mathbf{R}^n, \ \pi(x_1,..,x_{n+1}) = (x_1,..,x_n).
\end{equation}
Because $u_0$ is proper and convex, we can choose $i_0 = i_0(n, T, r, \lambda, b) > 0$ so that for all $i \geq i_0$, 
\begin{equation}
    \Tilde{\Sigma}_{-\Tilde{T}} + p^i \subset D^i_0.
\end{equation}
Then, by avoidance principle, we have
\begin{equation}
    \Tilde{\Sigma}_{-\Tilde{T} + t} + p^i \subset D^i_t
\end{equation}
for all $t \in [0,T].$ Together with the fact that $M^i_t$ are convex, and reflection symmetric with respect to $\{x_3 = i\}$, we see that $M^i_t$ is well defined up to $t= T$, and $\{\hat{M}^i_t\}_{t \in [0, T]}$ is graphical over $K_{r, \lambda}$. If we let $u^i$ so that $\hat{M}^i_t = \text{Graph}(u^i(\cdot, t))$, then note that since $u^i(x, 0) = u_0(x)$ by construction, and $T, r, \lambda$ are arbitrary, we can apply the apriori estimates obtained in section \ref{existence of graphical solution section}, in particular lemma \ref{higher order derivative estimate} to obtain uniform spacetime interior $C^k$ estimates of $u^i$ independently of $i$. This allows us to take a local smooth limit and obtain a graphical, complete longtime solution to mean curvature flow with $u_0$ as its initial data. Since each $\{M^i_t\}$ are compact, convex solutions to mean curvature flow, the graphical solution must also be convex, and Hamilton's Harnack inequality must hold as well.
\end{proof}

We are now ready to prove theorem \ref{dynamical stability of 2d translators}.
\begin{proof}[Proof of theorem \ref{dynamical stability of 2d translators}, remark \ref{independence of constants 2d translators remark}]
We first use lemma \ref{existence of convex evolution when deltawing} to construct the longtime graphical complete solution to mean curvature flow with $u_0$ as initial data. If $\overline{M}$ is a (tilted) grim reaper plane, then by lemma \ref{splitting off of intial data}, 
\begin{equation}
    u_0(x_1,x_2) = u_0(0, x_2) + x_1\tan\theta \text{ for }\theta = \arccos\frac{\pi}{2b}. 
\end{equation}
Thus, by possibly tilting the graph, and rescaling, we may without loss of generality assume that $\overline{M}$ is a grim reaper plane, in which case, the initial data must be of the form
\begin{equation}
    u_0(x_1,x_2) = u_0(0,x_2), \ -\ln \cos x_2 - C_0 \leq u_0(0, x_2) \leq -\ln\cos x_2 + C_0.
\end{equation}
Then, we may apply lemma \ref{existence of convex evolution when deltawing} to $u_0(x_2) = u_0(0, x_2)$ to obtain
\begin{equation}
    u(x_1,x_2,t) = u(x_2,t).
\end{equation}
In the case $\overline{M}$ is a $\Delta$-wing, then by directly applying lemma \ref{existence of convex evolution when deltawing} to $u_0$, we can construct the longtime solution $u$. Note that in both cases, $M_t = \text{Graph}(u(\cdot, t))$ is convex, complete, and the Harnack inequality holds. \\

We now follow our general strategy \ref{generalstrategy}. For each $t_j \to \infty$, we consider 
\begin{equation}
    t \to M^j_t = M_{t + t_j} - t_je_3, \ \ t \in [-t_j, \infty).
\end{equation}
Then by assumption \eqref{C0 closeness 2d}, if we let $M^j_t = \text{Graph}(u^j(\cdot, t))$, then
\begin{equation}\label{C0 closeness of limit flow 2d}
    \overline{u}(x) + t - C_0 \leq u^j(x, t) \leq \overline{u}(x) + t + C_0.
\end{equation}
Once again, the two translating solutions acts as barriers and provide a uniform $C^0$ estimate for $u^j$ independent of $j$. Then by lemma \ref{higher order derivative estimate}, we have uniform $C^k$ estimates for $u^j$, hence we can take a subsequential limit $M^{\infty}_t = \text{Graph}(u^{\infty}(\cdot, t))$. Let us abuse notation, and denote the subsequence by $t_j$.\\

It is clear that $M^{\infty}_t$ is complete, convex, graphical solution to mean curvature flow on which Hamilton's Harnack inequality holds. Thus, by applying Harnack inequality \eqref{ancient harnack inequality} with 
\begin{equation}
    V = Hve_3^{tan},
\end{equation}
where $H,v$ are given in definition \ref{def of commonly used quantities}, we have
\begin{equation}\label{acceleration of bulk velocity 2d}
    \partial_t^2u^{\infty} \geq 0.
\end{equation}
We claim that 
\begin{equation}
\partial_tu^{\infty} = 1.    
\end{equation}
If there exists $x_0 \in M^{\infty}_{t_0}$ so that
\begin{equation}
    \partial_tu^{\infty}(x_0, t_0) = 1 + \delta > 1,
\end{equation}
then by \eqref{acceleration of bulk velocity 2d}, we have
\begin{equation}
    \partial_tu^{\infty}(x_0, t) \geq  1 + \delta \text{ for all }t \geq t_0.
\end{equation}
By integrating in $t$, we have
\begin{equation}
    u^{\infty}(x_0, t) - u^{\infty}(x_0, s) \geq (1 + \delta)(t - s) \text{ for all }t > s \geq t_0.
\end{equation}
On the other hand, by \eqref{C0 closeness of limit flow 2d}, we have
\begin{equation}
    u^{\infty}(x_0, t) - u^{\infty}(x_0, s) \leq (t - s)  +  2C_0\text{ for all }t > s \geq t_0.
\end{equation}
which is a contradiction for large $t - s$. Thus $\partial_tu^{\infty}\leq 1$. The reverse inequality is obtained similarly. Hence we see that $\partial_tu^{\infty} = 1$, i.e $M^{\infty}_0$ is a graphical translator over $\Omega^2_b$. \eqref{C0 closeness of limit flow 2d} implies that $M^{\infty}_0$ is not contained in $\Omega^2_{b'} \times \mathbf{R}$ for $b' < b$. Then, by the classification result in \cite{Hoffman2018GraphicalTF}, $M^{\infty}_0$ is either (tilted) grim reaper plane, or a $\Delta$-wing defined over $\Omega^2_b$. It is clear that one type of translator cannot lie in between two copies of another, hence $M^{\infty} = \overline{M}$ up to translation. \\

Because of the upper and lower barriers
\begin{equation}\label{c0 bound for limit flow in 2d translator case}
    \overline{u}(x) - C_0 \leq u^{\infty}(x, 0) \leq \overline{u}(x) + C_0,
\end{equation}
there cannot be translation in $e_2$ direction. Thus, $u^{\infty}(\cdot, 0)$ must be of the form
\begin{equation}
    u^{\infty}(x_1, x_2, 0) = \overline{u}(x_1 + c_1, x_2) + c_0
\end{equation}
for some $c_0, c_1 \in \mathbf{R}$. \\

If $\overline{M}$ is a (tilted) grim reaper plane, then there is no loss of generality in taking $c_1 = 0$. Then by \eqref{c0 bound for limit flow in 2d translator case}, we have
\begin{equation}\label{subconvergence2dgrimreaperplane}
    \lim_{j\to \infty}u(x, t+t_j) - t_j = \overline{u}(x) + c_0 +t\text{ in }C^{\infty}_{loc}(\Omega^2_b\times \mathbf{R}) \text{ for some }c_0 \leq |C_0|.
\end{equation}
If $\overline{M}$ is a $\Delta$-wing, then we have
\begin{equation}\label{subconvergence2ddeltawing}
    \lim_{j\to \infty}u(x, t+t_j) - t_j = \overline{u}(x + c_1e_1) + c_0 +t\text{ in }C^{\infty}_{loc}(\Omega^2_b\times \mathbf{R}) \text{ for some }c_0, c_1 \in \mathbf{R}
\end{equation}
with
\begin{equation}\label{bound for amount of translation for deltawing}
    |c_0| + \tan\theta|c_1| \leq C_0 \text{ where }\theta = \arccos\frac{\pi}{2b}.
\end{equation}
The bound \eqref{bound for amount of translation for deltawing} is derived as follows. Assume the bound is not true. By reflection in $x_1$ variable, we may without loss of generality assume that $c_1 \geq 0$. We first consider when $c_0 \geq 0$. Then by theorem 1.5 in \cite{spruck2020complete} and \eqref{c0 bound for limit flow in 2d translator case}, we have
\begin{equation}
    \lim_{x_1 \to \infty}\overline{u}(x_1 + c_1, 0) + c_0 - \overline{u}(x_1, 0) = \tan\theta |c_1| + |c_0| \leq C_0
\end{equation}
which is a contradiction. The case when $c_0 \leq 0$ is dealt with similarly by looking $x_1 \to -\infty$ instead.\\

We now prove remark \ref{independence of constants 2d translators remark} by showing that $c_0, c_1$ are independent of choice of subsequence provided
\begin{equation}
    \limsup_{j \to \infty}|t_{j+1} - t_j| < \infty.
\end{equation}
Then there exists a fixed $\alpha > 0$ so that
\begin{equation}\label{bound on subsequence in 2d}
  |t_{j+1} - t_j| \leq \alpha \text{ for all }j.  
\end{equation}
 Since this part of the proof is essentially same for all types of translators, We only prove this for $\Delta$-wing. Choose another sequence $s_j\to \infty$, and let $\Tilde{c}_0, \Tilde{c}_1$ be the constants corresponding to $s_j$, i.e
 \begin{equation}
      \lim_{j\to \infty}u(x, t+s_j) - s_j = \overline{u}(x + \Tilde{c}_1e_1) + \Tilde{c}_0 +t\text{ in }C^{\infty}_{loc}(\Omega^2_b\times \mathbf{R}).
 \end{equation}
 Then set $\Tilde{t}_j$ to be one of $\{t_i\}_{i \in \mathbf{N}}$ which is closest to $s_j$. Then because of \eqref{bound on subsequence in 2d}, $-\alpha \leq s_j - \Tilde{t}_s \leq \alpha$, and $\Tilde{t}_j$ is a subsequence of $t_j$, \eqref{subconvergence2ddeltawing} implies 
\begin{align*}
    c_0 &= \lim_{j\to \infty}u(x, s_j) - \Tilde{t}_j -(s_j-\Tilde{t}_j) - \overline{u}(x+c_1e_1)\\& = \lim_{j \to \infty}u(x, s_j) - s_j - \overline{u}(x + c_1e_1) \\&= \overline{u}(x + \Tilde{c}_1e_1) -\overline{u}(x + c_1e_1)+ \Tilde{c}_0
\end{align*}
for all $x \in \Omega^2_b$. This immediately implies that $c_0 = \Tilde{c}_0$, $c_1 = \Tilde{c}_1$. Thus, the constants $c_0, c_1$ are independent of choice of subsequence, hence we get full convergence
\begin{equation}
    \lim_{t\to \infty}u(x,t) - t = \overline{u}(x + c_1e_1) + c_0 \text{ in }C^{\infty}_{loc}(\Omega_b^2),
\end{equation}
completing the proof of theorem \ref{dynamical stability of 2d translators}, remark \ref{independence of constants 2d translators remark}.
\end{proof}
\section{Dynamical stability of asymptotically cylindrical translators}\label{asymptoticallycylindricalsection}
In this section, we prove dynamical stability of asymptotically cylindrical translators (see definition \ref{def of asymptotically cylindrical translators.}). Note that such translators are fully classified by Bamler-Lai \cite{Bamler2025-pdeodi}, \cite{Bamler2025-thclassification}, and are either the bowl soliton, or entire graph translators constructed in \cite{Hoffman2018GraphicalTF} with possible Euclidean factors. The proof sketch is as follows. We first show that under the given conditions, the forward limit must be an asymptotically cylindrical flow. Then by using the classification result in \cite{Bamler2025-thclassification}, we can deduce that it must be a translating flow with the same velocity as the initially given translator. In the case the initial translator is a bowl soliton with possible Euclidean factors, one can use uniqueness result of bowl solitons \cite{Haslhofer2014UniquenessOT} to show that the two flows are the same up to translation. In the case of flying wing solutions, we show that the two flows are the same up to translation by using the parametrization of space of all non-collapsed translators given in theorem 1.6 in \cite{Bamler2025-thclassification}. \\

We first restate theorem \ref{dynamical stability of cylindrical translators intro} for the reader's convenience. 
\begin{theorem}\label{dynamical stability of asymptotically cylindrical translators section theorem}[Theorem \ref{dynamical stability of cylindrical translators intro}]
Let $n \geq 2$, and $1 \leq k \leq n-1$. Let $\overline{M} = \text{Graph}(\overline{u})$ be an asymptotically $(n,k)$-cylindrical translator (see definition \ref{def of asymptotically cylindrical translators.}). Let $u_0\in C^{2 + \alpha}(\mathbf{R}^n)$ so that
\begin{equation}\label{initial assumption on perturbation asympt cyll}
    (i)\  \|u_0 - \overline{u}\|_{C^0(\mathbf{R}^n)} \leq C_0 < \infty, \ \ \ (ii)\  M_0 = \text{Graph}(u_0) \text{ is mean convex}.
\end{equation}
Then there exists $u \in C^{\infty}(\mathbf{R}^n \times (0, \infty))\cap C^0(\mathbf{R}^n \times [0, \infty))$ so that $u$ solves the initial value problem
\begin{equation}
    \begin{cases}
    \frac{\partial u}{\partial t} = \sqrt{1 + |Du|^2}\textit{\emph{div}}(\frac{Du}{\sqrt{1 + |Du|^2}}) \text{ in }\mathbf{R}^n \times (0, \infty) \\ u(\cdot, 0) = u_0 \text{ in }\mathbf{R}^n.
    \end{cases}
\end{equation}
Furthermore, for each $t_j \to \infty$, there exists subsequence $\{\Tilde{t}_j\} \subset \{t_j\}$ so that
\begin{equation}
    \lim_{j \to \infty}u(x, t + \Tilde{t}_j) - \Tilde{t}_j = \overline{u}(x + p_1)+t + p_2 \text{ in }C^{\infty}_{loc}(\mathbf{R}^n\times \mathbf{R}),
\end{equation}
where $(p_1, p_2) \in \mathbf{R}^n \times \mathbf{R} = \mathbf{R}^{n+1}$. In other words, the perturbed solution converges locally smoothly to the translating solution $\overline{M} + te_{n+1}$ modulo space translation. The amount of translation $(p_1, p_2)$ has norm comparable to $C_0$.
\end{theorem}
\begin{remark}
    Note that all asymptotically cylindrical translators are convex, hence mean convex. Thus mean convexity of $M_0$ in theorem \ref{dynamical stability of asymptotically cylindrical translators section theorem} can be viewed as a $C^2$ closeness condition to $\overline{u}$. 
\end{remark}
\begin{remark}\label{independence of constants to subsequence cylindrical remark}
As in the case of theorem \ref{dynamical stability of 2d translators}, it is unclear in general if $(p_1, p_2)$ is independent of choice of (sub)sequence. However, if there exists $\{\Tilde{t}_j\}$ so that
\begin{equation}
    \limsup_{j \to \infty}|\Tilde{t}_{j+1} - \Tilde{t}_j| < \infty,
\end{equation}
then we can show that the constants are independent of choice of (sub)sequence, hence we have full convergence
\begin{equation}
    \lim_{t \to \infty}u(x,t) - t = \overline{u}(x + p_1) + p_2 \text{ in }C^{\infty}_{loc}(\mathbf{R}^n).
\end{equation}
\end{remark}
We first prove corollary \ref{strong convergence to bowl with strong asymptotic behavior}, assuming theorem \ref{dynamical stability of asymptotically cylindrical translators section theorem}. The idea is essentially the same as \cite{bowlstability}, except that we use theorem \ref{dynamical stability of asymptotically cylindrical translators section theorem} to avoid the use of the strong maximum principle in lemma 4.2 in \cite{bowlstability}.
\begin{proof}[Proof of corollary \ref{strong convergence to bowl with strong asymptotic behavior}]
By theorem \ref{dynamical stability of asymptotically cylindrical translators section theorem}, for each $t_j \to \infty$, there exists subsequence $\{\Tilde{t}_j\} \subset\{t_j\}$, and constants $(p_1, p_2) \in \mathbf{R}^n \times \mathbf{R}$ so that
\begin{equation}\label{longtimebehavior}
    \lim_{j \to +\infty}u(x,t + \Tilde{t}_j) - \Tilde{t}_j = \overline{u}(x+p_1) + t+p_2 \text{ in }C^{\infty}_{loc}(\mathbf{R}^n\times \mathbf{R}).
\end{equation}
We first show that $p_1 = 0$. Writing $p_1 = (\hat{p}_1, \Tilde{p}_1) \in \mathbf{R}^{n-k+1}\times \mathbf{R}^{k-1}$, because $\overline{u}$ is independent of the last $k-1$ component, we may simply put $\Tilde{p}_1 = 0$. If $\hat{p}_1 \neq 0$, then we have for each $x = (\hat{x}, 0) \in \mathbf{R}^n$, 
\begin{equation}
    |u(x,\Tilde{t}_j ) - \Tilde{t}_j - \overline{u}(x + \hat{p}_1)| \leq C , \ \ \ |u(x,\Tilde{t}_j) - \Tilde{t}_j - \overline{u}(x)| \leq C
\end{equation}
for some fixed $C > 0$ for all sufficiently large $j$. The first inequality is due to \eqref{longtimebehavior} and the second inequality is due to avoidance principle together with assumption \eqref{strongasymptoticbehaviorassumption}. This implies that 
\begin{equation}\label{bound for bowlsoliton}
    |\overline{u}(x + \hat{p}_1) - \overline{u}(x)| \leq 2C < \infty.
\end{equation}
On the other hand, by recalling the asymptotic behavior of the bowl soliton as $|x| \to \infty$ which is proved in \cite{bowlstability}, we have
\begin{equation}\label{asymptotic behavior of bowl}
    \overline{u}(x) = \overline{u}(\hat{x}, 0)= \frac{|\hat{x}|^2}{2(n-k)} - \ln |\hat{x}| + C + O(|\hat{x}|^{-1}) \text{ as }|\hat{x}| \to \infty.
\end{equation}
This gives a contradiction to inequality \eqref{bound for bowlsoliton} for large $|\hat{x}|$ due to the quadratic growth of $\overline{u}$, hence $\hat{p}_1 = 0$.\\

We now show that with the extra assumption \eqref{strongasymptoticbehaviorassumption}, $p_2 = 0$. To do so, we use the barriers constructed in \cite{bowlstability}.
The assumption \eqref{strongasymptoticbehaviorassumption} implies that for each $\epsilon > 0$, there exists $R > 0$ so that for all $x = (\hat{x}, \Tilde{x})$ with $|\hat{x}| \geq R$
\begin{equation}
    \overline{u}(x) - \epsilon < u_0(x) < \overline{u}(x) + \epsilon.
\end{equation}
Then, we can use the upper/lower half wings constructed in lemma 2.3 in \cite{bowlstability} to construct lower/upper barriers for $u_0$. More specifically, we consider $W^{+}_R\times \mathbf{R}^{k-1}$, $W^{-}_R\times \mathbf{R}^{k-1}$ with $W^{+}_{R}$, $W^-_{R}$ being the `$\epsilon$-shifted $n-k+1$-dimensional lower / upper half wing solutions' in $\mathbf{R}^{n-k+2}$ defined in lemma 4.1 in \cite{bowlstability} with $n-k+1$ instead of $n$. Then by avoidance principle and the fact that $W^{+}_R$, $W^{-}_R$ are asymptotic to $n-k+1$-dimensional bowl soliton as $|\hat{x}| \to \infty$, we see that by possibly making $R > 0$ larger, we have
\begin{equation}\label{asmpytoticclosenessfarried}
    \sup_{t \geq 0, |\hat{x}| \geq R, \Tilde{x}\in \mathbf{R}^{k-1}} |u(x,t) - \overline{u}(x) - t| \leq 2\epsilon.
\end{equation}
Combining with \eqref{longtimebehavior}, this implies that
\begin{equation}
    |p_2| \leq 2\epsilon.
\end{equation}
Since $\epsilon > 0$ is arbitrary, we have $p_2 = 0$, hence proving corollary \ref{strong convergence to bowl with strong asymptotic behavior}.
\end{proof}
We now prove theorem \ref{dynamical stability of asymptotically cylindrical translators section theorem}.
\begin{proof}[Proof of theorem \ref{dynamical stability of asymptotically cylindrical translators section theorem}, remark \ref{independence of constants to subsequence cylindrical remark}]
For convenience, we set 
\begin{equation}
    \overline{M}_t = \overline{M} + te_{n+1}, \ \ t\in \mathbf{R}
\end{equation}
to be the translating solution generated by initially given translator $\overline{M}$. Moreover, by theorem \ref{parametrization of space of solitons by bamler lai}, we may suitably translate $\overline{u}$ and $u_0$ in space so that $\overline{M} \in \textbf{MCF}^{n,k}_{\text{soliton}}$ (see definition \ref{Definition of MCFsoliton}).\\

By classical regularity theory of Ecker and Huisken \cite{eckerhuisken}, \cite{Ecker1991InteriorEF}, one obtains a longtime solution to graphical mean curvature flow with $u_0$ as the initial data by taking the limit of solutions to the following initial-boundary value problem
\begin{equation}
    \begin{cases}
    \frac{\partial u}{\partial t} = \sqrt{1 + |Du|^2}\textit{\emph{div}}(\frac{Du}{\sqrt{1 + |Du|^2}}) \text{ in }B_{\mathbf{R}^n}(0, R) \times (0, \infty) \\ u = u_0 \text{ in }\mathbf{R}^n\times \{0\} \cup \partial B_{\mathbf{R}^n}(0, R)\times (0, \infty)
    \end{cases}
\end{equation}
with $R = 1,2,3,..$. Moreover, since $M_0 = \text{Graph}(u_0)$ is mean convex, by weak maximum principle, we see that the solution we obtain is mean convex for all later time. Let us denote the solution to be 
\begin{equation}
    M_t = \text{Graph}(u(\cdot, t)).
\end{equation}
We now consider the forward limit of the flow, namely for any $t_j \to + \infty$, define
\begin{equation}
    M^j_t = M_{t + t_j} - t_je_{n+1}.
\end{equation}
First note that because of (i) in \eqref{initial assumption on perturbation asympt cyll}, we have
\begin{equation}
    \overline{u}(x) - C_0 \leq u_0(x) \leq \overline{u}(x) + C_0.
\end{equation}
Therefore by avoidance principle \cite{White2024TheAP}, we have
\begin{equation}
    \overline{u}(x) + t - C_0 \leq u(x,t) \leq \overline{u}(x) +t+ C_0
\end{equation}
for all $t \geq 0$. In particular, since
\begin{equation}
    M^j_t = \text{Graph}(u^j(\cdot, t)) \text{ with }u^j(x,t) = u(x, t_j + t) - t_j,
\end{equation}
we have
\begin{equation}\label{C0 bound on appx seq asym cyl}
    \overline{u}(x) + t - C_0 \leq u^j(x,t) \leq \overline{u}(x) +t+ C_0
\end{equation}
Since we have uniform local $C^0$ bound on $u^j$ independent of $j$, by interior regularity theory \cite{Ecker1991InteriorEF}, we also have uniform local $C^m$ bound on $u^j$ independent of $j$. Thus, by possibly passing through subsequence, we can take a local smooth limit
\begin{equation}
    u^{\infty}(x,t) = \lim_{j \to \infty}u^j(x,t)
\end{equation}
which is defined on $\mathbf{R}^n \times \mathbf{R}$. Then the graph
\begin{equation}
    M^{\infty}_t = \text{Graph}(u^{\infty}(\cdot, t))
\end{equation}
is an eternal solution to mean curvature flow. Since the inequality \eqref{C0 bound on appx seq asym cyl} passes through the limit, we have
\begin{equation}\label{C0 bound on forward limit asym cyl} 
    \overline{u}(x) + t - C_0 \leq u^{\infty}(x,t) \leq \overline{u}(x) + t + C_0,
\end{equation}
i.e $M^{\infty}_t$ is trapped between two copies of $\overline{M}_t$.\\

We now claim that $M^{\infty}_t$ is a non-collapsed, translator of velocity $e_{n+1}$. This is immediate once we show that $
\{M^{\infty}_t\}_{t \in (-\infty, \infty)}$ is an asymptotically $(n,k)$-cylindrical flow. Once this is shown, by the recent classification of Bamler-Lai \cite{Bamler2025-thclassification} (see theorem \ref{parametrization of space of solitons by bamler lai}), $M^{\infty}_t$  must be a non-collapsed translating solution. The fact that the velocity is equal to $e_{n+1}$ readily follows from \eqref{C0 bound on forward limit asym cyl}. For example, one can use Hamilton's Harnack inequality \eqref{ancient harnack inequality} as we did in the proof of previous theorems. (Note that we can apply Harnack inequality on asymptotically cylindrical translators since they are complete hypersurfaces with uniformly bounded curvature). \\

To show that $M^{\infty}_t$ is an asymptotically $(n,k)$-cylindrical flow, we follow the ideas of White \cite{White2000TheSO}. We consider the parabolic blowdown of the limit flow, namely for $\lambda \to 0$, define
\begin{equation}
    N^{\lambda}_t = \lambda M^{\infty}_{\lambda^{-2}t} \text{ for }t \in (-\infty, 0).
\end{equation}
First note that since $M_0$ is mean convex, so is $M^{\infty}_t$, and $N^{\lambda}_t$. Moreover, since $N^{\lambda}_t$ is an entire graph, the collection of all translations of $N^{\lambda}_t$ in $e_{n+1}$ direction foliates $\mathbf{R}^{n+1}$. Therefore, by standard calibration argument (see \cite{Haslhofer2013MeanCF}, remark 2.6), we see that $N^{\lambda}_t$ has polynomial volume growth, i.e
\begin{equation}
    \frac{\mathcal{H}^n(N^{\lambda}_t\cap B(x_0,r))}{r^n} \leq c(n) \text{ for all }B(x_0, r) \subset \mathbf{R}^{n+1},
\end{equation}
where $\mathcal{H}^n$ is the $n$-dimensional Hausdorff measure in $\mathbf{R}^{n+1}$. Therefore, one can apply compactness theorem for integral Brakke flows to conclude that there exists a self shrinking tangent flow
\begin{equation}
    \mu_t = \lim_{\lambda \to 0}N^{\lambda}_t
\end{equation}
where the convergence is currently at the level of integral Brakke flow sense. \\

We claim that the tangent flow is a multiplicity one shrinking cylinder $M^{n,k}_{t} = \mathbf{S}^{n-k}(\sqrt{-2(n-k)t})\times \mathbf{R}^k$. We first look at the support of $\mu_t$. Take any $\phi \in C_c^0(\mathbf{R}^{n+1})$ so that 
\begin{equation}
    \text
    {supp}(\phi) \cap M^{n,k}_t = \emptyset.
\end{equation}
Because $M^{\infty}_t$ is trapped between $\overline{M}^+_t = \overline{M}_t + C_0e_{n+1}$ and $\overline{M}^-_t = \overline{M}_t   - C_0e_{n+1}$ due to \eqref{C0 bound on forward limit asym cyl}, the same containment relation remains true after parabolically rescaling by $\lambda$. By assumption that $\overline{M}$ is an asymptotically $(n,k)$-cylindrical translator, we see that both $\overline{M}^+_t, \overline{M}^-_t$ has multiplicity one $M^{n,k}_t$ as its tangent flow at infinity. Therefore by using the rescaled $\overline{M}^+_t, \overline{M}^-_t$ as barriers, we immediately see that for all sufficiently small $\lambda > 0$, 
\begin{equation}
    N^{\lambda}_t \cap \text{supp}(\phi) = \emptyset,
\end{equation}
which implies that 
\begin{equation}
    \text{supp}(\mu_t) \cap \text{supp}(\phi) = \emptyset.
\end{equation}
Since $\phi$ is arbitrary function with support away from $M^{n,k}_t$, 
this implies that 
\begin{equation}
    \text{supp}(\mu_t) \subset M^{n,k}_t.
\end{equation}

To show reverse inclusion, we again use calibration argument. Take any $x_0 \in M^{n,k}_t$, and $r > 0$. We claim that 
\begin{equation}
    \frac{\mu_t(\overline{B(x_0,r)})}{r^n} > 0 
\end{equation}
First note that 
\begin{equation}
    \{\lambda\overline{M}_{\lambda^{-2}t} + ae_{n+1}\}_{a \in \mathbf{R}}
\end{equation}
is a convex foliation of $\mathbf{R}^{n+1}$. Estimate \eqref{C0 bound on forward limit asym cyl} implies that $N^{\lambda}_t$ is trapped between $\lambda\overline{M}_{\lambda^{-2}t} \pm C_0\lambda e_{n+1}$. Since $N^{\lambda}_t$ is an entire graph, and $\lambda\overline{M}_{\lambda^{-2}t} \pm C_0\lambda e_{n+1} \to M^{n,k}_t$ locally smoothly, we can apply standard calibration argument (see \cite{Haslhofer2013MeanCF}, remark 2.6) to compare 
\begin{equation}
\mathcal{H}^n(\lambda\overline{M}_{\lambda^{-2}t} + C_0\lambda e_{n+1}\cap \overline{B(x_0,r)})
\end{equation}
and
\begin{equation}
\mathcal{H}^n(N^{\lambda}_t \cap \overline{B(x_0,r)}) + \mathcal{H}^n(S_{\lambda}\cap \overline{B(x_0,r)})
\end{equation}
where $S_{\lambda}$ is the set of all points on $\partial B(x_0, r)$ which lie between $\lambda\overline{M}_{\lambda^{-2}t} \pm C_0\lambda e_{n+1}$. This implies that
\begin{equation}
  \mathcal{H}^n(N^{\lambda}_t \cap \overline{B(x_0,r)}) + \mathcal{H}^n(S_{\lambda}\cap \overline{B(x_0,r)}) \geq \mathcal{H}^n(\lambda\overline{M}_{\lambda^{-2}t} + C_0\lambda e_{n+1}\cap B(x_0,r)).
  \end{equation}
Therefore, by letting $\lambda \to 0$ and using $x_0 \in M^{n,k}_t$, $\lambda\overline{M}_{\lambda^{-2}t} \pm \lambda C_0 e_{n+1}\to M^{n,k}_t$, we see that 
\begin{equation}
    \mathcal{H}^n(S_{\lambda}\cap \overline{B(x_0, r)}) = \mathcal{H}^n(S_{\lambda}) \to 0 \text{ as }\lambda \to 0,
\end{equation}
thus
\begin{equation}\label{area compartison}
    \mu_t(\overline{B(x_0, r)}) \geq \mathcal{H}^n(M^{n,k}_t \cap B(x_0, r)) > 0.
\end{equation}
This implies that $x_0 \in \text{supp}(\mu_t)$, hence $M^{n,k}_t = \text{supp}(\mu_t)$. \\

Thus, $\mu_t$ is equal to $M^{n,k}_t$ with possible multiplicities. By once again following the ideas of White \cite{White2000TheSO}, we can rule out high multiplicities. Since the argument is straightforward and is very similar to the previous argument, we record it here in detail. Note that $N^{\lambda}_t$ is also a mean convex entire graph. Therefore, we can repeat the previous calibration argument with role of $N^{\lambda}_t$ and $\lambda \overline{M}_{\lambda^{-2}t}$ reversed. More specifically, note that 
\begin{equation}
    \{ N^{\lambda}_t + ae_{n+1}\}_{a \in \mathbf{R}}
\end{equation}
is a mean convex foliation of entire $\mathbf{R}^{n+1}$. Choose any $x_0 \in M^{n,k}_{t}$. Since $N^{\lambda}_t$ is trapped between $\lambda\overline{M}_{\lambda^{-2}t} \pm \lambda C_0 e_{n+1}$, by standard calibration argument, we can compare
\begin{equation}
    \mathcal{H}^n(N^{\lambda}_t \cap \overline{B(x_0, r)})
\end{equation}
and
\begin{equation}
\mathcal{H}^n(\lambda\overline{M}_{\lambda^{-2}t} - C_0\lambda e_{n+1}\cap \overline{B(x_0, r)}) + \mathcal{H}^n(S_{\lambda}),
\end{equation}
where $S_{\lambda}$ is the set of all points on $\partial B(x_0, r)$ which lie between $\lambda \overline{M}_{\lambda^{-2}t} \pm C_0\lambda e_{n+1}$. This gives us
\begin{equation}
    \mathcal{H}^n(\lambda\overline{M}_{\lambda^{-2}t} - C_0\lambda e_{n+1}\cap \overline{B(x_0, r)}) + \mathcal{H}^n(S_{\lambda}) \geq \mathcal{H}^n(N^{\lambda}_t \cap \overline{B(x_0, r)}).
\end{equation}
Then after taking the limit $\lambda \to 0$, since $\mathcal{H}^n(S_{\lambda}) \to 0$, and $\overline{M}$ is an asymptotically cylindrical translator, we obtain a reverse inequality to \eqref{area compartison}, namely
\begin{equation}
 \mu_t(B(x_0, r)) \leq \mathcal{H}^n(M^{n,k}_t \cap B(x_0, r))   
\end{equation}
for all $r > 0$. Thus, by dividing by $\omega_nr^n$ where $\omega_n$ is the volume of an $n$-dimensional unit ball, letting $r \to 0$ and using $x_0 \in M^{n,k}_t$, we see that the $n$-dimensional Hausdorff density of $\mu_t$ at $x_0 \in M^{n,k}_t$ is bounded by 1. Since $\mu_t$ is an integral $n$-varifold with support $M^{n,k}_t$, this implies that $\mu_t$ is equal to multiplicity one $M^{n,k}_t$. Thus by local regularity theorem \cite{White2005-nl}, the convergence is smooth, hence $M^{\infty}_t$ is an asymptotically $(n,k)$-cylindrical flow, hence it must be a non-collapsed, convex translator of velocity $e_{n+1}$. In particular, we can write
\begin{equation}
    u^{\infty}(x,t) = u^{\infty}(x,0) + t = u^{\infty}(x) + t.
\end{equation}

We first consider the case when $  \overline{M}$ is a $n-k+1$ dimensional Bowl soliton times $\mathbf{R}^{k-1}$. This means 
\begin{equation}
    \overline{u}(x_1,..,x_n) = \overline{u}(x_1,..,x_{n-k+1}).
\end{equation}
The convexity of $u^{\infty}(x)$ together with
\begin{equation}
    \overline{u} - C_0 \leq u^{\infty}(x) \leq \overline{u}(x) + C_0
\end{equation}
implies that $u^{\infty}$ is also independent of $(x_{n-k+2},..,x_n)$, i.e $M^{\infty}$ splits off the same $\mathbf{R}^{k-1}$ factor as $\overline{M}$. Then, we see that the lower dimensional translator factor of $M^{\infty}$ is asymptotic to the round cylinder 
\begin{equation}
    \mathbf{S}^{n-k}(\sqrt{2(n-k)})\times \mathbf{R} 
\end{equation}
in $\mathbf{R}^{n-k+2}$. By uniqueness of bowl soliton proved in \cite{Haslhofer2014UniquenessOT}, the lower dimensional translator factor must also be a $n-k+1$-dimensional bowl soliton, hence we see that 
\begin{equation}
    M^{\infty} = \overline{M} + p_0
\end{equation}
for some $p_0 \in \mathbf{R}^{n+1}$.\\

We now consider the remaining case when $\overline{M}$ is a flying wing soliton constructed in \cite{Hoffman2018GraphicalTF} with possible Euclidean factor . In this case, $M^{\infty}$ cannot be a bowl soliton with a Euclidean factor. Indeed, if $M^{\infty}$ was a bowl soliton with a Euclidean factor, then by reversing the role of $M^{\infty}$ and $\overline{M}$, and viewing $M^{\infty}$ as the barrier, previous argument implies that $\overline{M}$ also has to be a bowl soliton with a Euclidean factor which is not the case. Thus, both $M^{\infty}$ and $\overline{M}$ are flying wing solitons with possible Euclidean factors, in particular we can apply proposition \ref{asymptotic bevaior of nkflows} to both $M^{\infty}_t$ and $\overline{M}_t$.\\

We now claim that $M^{\infty} = \overline{M} + p_0$ for some constant $p_0 \in \mathbf{R}^{n+1}$. First, by theorem \ref{parametrization of space of solitons by bamler lai}, there exists some $p_0 \in \mathbf{R}^{n+1}$ so that the translated flow $\hat{M}_t = M^{\infty}_t + p_0$ belongs to $\textbf{MCF}^{n,k}_{\text{soliton}}$. Then by theorem \ref{parametrization of space of solitons by bamler lai}, we can show that $\hat{M} = \overline{M}$ by showing that $Q(\hat{M}) = Q(\overline{M})$, and $b(\hat{M}) = b(\overline{M})$ (see proposition \ref{asymptotic bevaior of nkflows}, definition \ref{definition of parameter Q,b in bamler lai}). Since both $\hat{M}$, and $\overline{M}$ have velocity $e_{n+1}$, by definition \ref{definition of parameter Q,b in bamler lai}, we see that
\begin{equation}
    b(\hat{M}) = \frac{e_{n+1}}{|e_{n+1}|^2} = e_{n+1} = b(\overline{M}).
\end{equation}
Thus it remains to show that 
\begin{equation}
    Q(\hat{M)} = Q(\overline{M}).
\end{equation}
The rough idea is that $Q$ determines the asymptotic expansion of the graph function of the rescaled mean curvature flow over the cylinder $M^{n,k}_{-1} = \mathbf{S}^{n-k}(\sqrt{2(n-k)})\times \mathbf{R}^k$ as $\tau \to -\infty $ up to an error of order $|\tau|^{-3}$ where $\tau$ is the rescaled time parameter (see proposition \ref{asymptotic bevaior of nkflows}). On the other hand, if we denote
\begin{equation}
    \hat{N}_{\tau} = e^{\tau/2}\hat{M}_{-e^{-\tau}}, \ \ \ \overline{N}_{\tau} = e^{\tau/2}\overline{M}_{-e^{-\tau}},
\end{equation}
then because before rescaling, $\hat{M}$ is trapped between two copies of $\overline{M}$ which are apart from each other by a fixed amount for all time slices, similar containment relation also holds after type I rescaling. However, because of the rescaling factor, we see that $\hat{N}_{\tau}$ is trapped between two copies of $\overline{N}_{\tau}$ which are apart from each other by at most $C_0e^{\tau/2}$. In particular, this implies that the difference between two graph functions must be exponentially decaying. This implies that the leading order terms must in fact coincide, i.e $Q(\hat{M}) = Q(\overline{M})$.\\

We make above argument precise. Recalling the definition of $\hat{M}$, and \eqref{C0 bound on forward limit asym cyl}, we see that for each $t \in \mathbf{R}$, $\hat{M}_t$ is trapped between
$\overline{M}_t + p_0 -C_0e_{n+1}$, and $\overline{M}_t + p_0 + C_0e_{n+1}$. Therefore, after type I rescaling, we see that $\hat{N}_{\tau}$ is trapped between $\overline{N}_{\tau} + e^{\tau/2}(p_0 - C_0e_{n+1})$ and $\overline{N}_{\tau} + e^{\tau/2}(p_0 + C_0e_{n+1})$. Note that the two barriers $\overline{N}_{\tau} + e^{\tau/2}(p_0 - C_0e_{n+1})$ and $\overline{N}_{\tau} + e^{\tau/2}(p_0 + C_0e_{n+1})$ are space translations of $\overline{N}_{\tau}$ with the norm of the translation vector bounded by $C_1e^{\tau/2}$ for some fixed uniform constant $C_1$ for all $\tau$. By combining this with the fact that $\hat{N}_{\tau}, \overline{N}_{\tau}$ converges to $M^{n,k}_{-1}$ locally smoothly, we see that for each fixed $R > 0$, we have
\begin{equation}\label{exponential decay of difference of graph function}
    \|\hat{u}(\cdot, \tau) - \overline{u}(\cdot, \tau)\|_{C^0(D_R)} \leq 10C_1e^{\tau/2}
\end{equation}
for all sufficiently small $\tau$, where $D_R = B_{\mathbf{R}^n}(0, R) \cap M^{n,k}_{-1}$.\\

We now apply proposition \ref{asymptotic bevaior of nkflows} to both $\hat{M}$ and $\overline{M}$. Then by possibly taking smaller $\Tilde{\tau}$, there exists $\hat{U}^+ = \hat{U}_1 + \hat{U}_{1/2} + \hat{U}_0 + .. + \hat{U}_{-10},\  \overline{U}^+ = \overline{U}_1 + \overline{U}_{1/2}+ \overline{U}_0 + .. + \overline{U}_{-10}$ with each $\hat{U}_{\lambda}, \overline{U}_{\lambda} \in \mathcal{V}_{\lambda} \subset L^2_{w}(M^{n,k}_{-1})$, so that for each $R > 0$
\begin{equation}\label{asymptotic expansion for each graph function}
    \|\hat{u}(\cdot \tau) - \hat{U}^+(\tau)\|_{C^{0}(D_R)}, \|\overline{u}(\cdot \tau) - \overline{U}^+(\tau)\|_{C^{0}(D_R)} \leq \frac{1}{10}(\Tilde{\tau} - \tau + 10)^{-11}
\end{equation}
for all sufficiently small $\tau \leq \Tilde{\tau}.$ Here, $L^2_w$ denotes the usual Gaussian weighted $L^2$ space over $M^{n,k}_{-1}$, and $\mathcal{V}_{\lambda}$ denote the $\lambda$-eigenspace of the linearized rescaled mean curvature flow operator at $M^{n,k}_{-1}$. Then, by combining \eqref{exponential decay of difference of graph function} and \eqref{asymptotic expansion for each graph function}, we see that
\begin{equation}\label{diff of asymptotic expansion}
    \|\hat{U}^+(\tau) - \overline{U}^+(\tau)\|_{C^0(D_R)} \leq C|\tau|^{-11}
\end{equation}
for each $R > 0$ for all sufficiently small $\tau$. Here, $C$ is independent of $R > 0$ and $\tau$.\\

Since both $\hat{U}^+(\tau)$, and $\overline{U}^+(\tau)$ are just finite sum of the eigenfunctions of the linearized rescaled mean curvature flow operator which are Hermite polynomials, we see that \eqref{diff of asymptotic expansion}, together with orthogonality of the eigenfunctions, imply that
\begin{equation}\label{diff of neutralmode}
    \|\hat{U}_0(\tau) - \overline{U}_0(\tau)\|_{L^2_w} \leq C|\tau|^{-11}
\end{equation}
for all sufficiently small $\tau$. Moreover, again by proposition \ref{asymptotic bevaior of nkflows}, we can find two smooth curves in $\mathcal{V}_{0}$ denoted by $\hat{V}(\tau)$, $\overline{V}(\tau)$ so that (i) $\hat{V}(\tau), \overline{V}(\tau)\to 0 \text{ as }\tau \to -\infty$, (ii) the coefficients of $\hat{V}(\tau), \overline{V}(\tau)$ are longtime solutions to the ODE \eqref{the ode solved by appx asymptotic expansion},  
\begin{equation}\label{Q polynomial determined}
    \textup{(iii)}\ \ \|\hat{U}_0(\tau) - \hat{V}(\tau)\|_{L^2_{w}}, \|\overline{U}_0(\tau) - \overline{V}(\tau)\|_{L^2_w} \leq C|\tau|^{-3}
\end{equation}
for all sufficiently small $\tau$, and (iv) $\hat{V}(\tau), \overline{V}(\tau)$ determines $Q(\hat{M)}, Q(\overline{M})$ respectively. Then by combining \eqref{diff of neutralmode} and \eqref{Q polynomial determined}, we see that 
\begin{equation}\label{cubic bound on diff of Q-poly}
    \|\hat{V}(\tau) - \overline{V}(\tau)\|_{L^2_w} \leq C|\tau|^{-3}
\end{equation}
for all sufficiently small $\tau$. Then by lemma \ref{key property of ODE solution}, \eqref{cubic bound on diff of Q-poly} implies that $\hat{V}(\tau) = \overline{V}(\tau)$, hence $Q(\hat{M}) = Q(\overline{M})$. Therefore, we ultimately have $Q(\hat{M}) = Q(\overline{M})$, and $b(\hat{M}) = b(\overline{M})$, hence by theorem \ref{parametrization of space of solitons by bamler lai}, we have $\hat{M} = \overline{M}$, hence 
\begin{equation}
    M^{\infty}_t = \overline{M} + te_{n+1} + p_0
\end{equation}
for some fixed $p_0 \in \mathbf{R}^{n+1}$. Writing $p_0 = (-p_1, p_2) \in \mathbf{R}^n \times \mathbf{R}$, this is equivalent to
\begin{equation}
    \lim_{j \to +\infty}u(x, t + t_j) - t_j = \overline{u}(x + p_1) + t + p_2.
\end{equation}
To show that the norm of $p_0 = (p_1, p_2)$ is comparable to $C_0$, first note that if $\overline{M}$ (hence $M^{\infty}$) has a Euclidean factor, we may simply take the components of $p_1$ in that factor to be $0$. Thus we might as well assume that $\overline{M}$ is strictly convex. Then the comparability of $|p_0|$ and $C_0$ follows from the strict convexity of $\overline{M}$, and the fact that $M^{\infty}_t = \overline{M} + te_{n+1} + p_0$ has to be trapped between 
$\overline{M} + te_{n+1} + C_0e_{n+1}$ and $\overline{M} + te_{n+1} - C_0e_{n+1}$.\\

Finally, we prove remark \ref{independence of constants to subsequence cylindrical remark}. Assume that 
\begin{equation}
    \lim_{j \to +\infty}u(x, t + t_j) - t_j = \overline{u}(x + p_1) + t + p_2.
\end{equation}
for some $p_0 = (p_1, p_2)$ with 
\begin{equation}\label{BOUND ON SUBSEQUENCE CYL}
    |t_{j+1} - t_j| \leq \alpha \text{ for all }j\text{ for some }\alpha > 0.
\end{equation}
Choose any other subsequence $s_j$, and let $\Tilde{p}_0 = (\Tilde{p}_1, \Tilde{p}_2)$ be the constants corresponding to $s_j$. Define $\Tilde{t}_j$ to be one of $\{t_i\}_{i \in \mathbf{N}}$ which is closest to $s_j$. Then because of \eqref{BOUND ON SUBSEQUENCE CYL}, we have $-\alpha \leq s_j - \Tilde{t}_j \leq \alpha$, hence the local smooth convergence implies 
\begin{align*}
     p_2 &= \lim_{j \to \infty}u(x, s_j-\Tilde{t}_j + \Tilde{t}_j) - s_j - \overline{u}(x + p_1) \\& = \lim_{j \to \infty}u(x, s_j) - s_j - \overline{u}(x + p_1) \\ & = \overline{u}(x + \Tilde{p}_1) - \overline{u}(x+p_1) + \Tilde{p}_2
\end{align*}
for all $x \in \mathbf{R}^n$. By taking some components of $p_1$, $\Tilde{p}_1$ to vanish whenever $\overline{M}$ has a Euclidean factor, strict convexity of $\overline{u}$ with respect to remaining components imply that $p_1 = \Tilde{p}_1$, and $p_2 = \Tilde{p}_2$. This shows that the constants are independent of choice of (sub)sequence, hence we get full convergence
\begin{equation}
    \lim_{t \to \infty}u(x, t) - t= \overline{u}(x + p_1) + p_2 \text{ in }C^{\infty}_{loc}(\mathbf{R}^n).
\end{equation}
\end{proof}
We now prove corollary \ref{alternative dynamical stability}. The proof is a slight modification of proof of theorem \ref{dynamical stability of asymptotically cylindrical translators section theorem}, so we only point out the necessary changes.
\begin{proof}[Proof of corollary \ref{alternative dynamical stability}]
The proof is essentially that of theorem \ref{dynamical stability of asymptotically cylindrical translators section theorem} with slight modifications when proving that $M^{\infty}_t$ is asymptotically $(n,k)$-cylindrical. First, by the entropy bound in (ii) of \eqref{alternative assumption on initial data cylindrical flows}, together with the fact that entropy is non-increasing under mean curvature flow, we see that 
\begin{equation}\label{entropy bound}
    \lambda(M^{\infty}_t) \leq \lambda(M_0) < 2\lambda(\overline{M}) \text{ for all }t \in \mathbf{R}.
\end{equation}
This implies that one can still consider the tangent flow at infinity
\begin{equation}\label{tangent flow at infinity alternative}
    \mu_t = \lim_{\lambda \to 0}N^{\lambda}_t,
\end{equation}
with $N^{\lambda}_t = \lambda M^{\infty}_{\lambda^{-2}t}$. $\mu_t$ is currently just an integral $n$-Brakke flow. This tangent flow is self shrinking flow, i.e. $\mu_t = \sqrt{-t}\mu_{-1}$, with $\mu_{-1}$ being the pushforward of a $F$-stationary integral varifold, denoted by $V(-1)$.
\\

By following the same arguments in proof of theorem \ref{dynamical stability of asymptotically cylindrical translators section theorem}, we can show that
\begin{equation}
    \textup{supp}(\mu_t) = M^{n,k}_t.
\end{equation}
Note that in this part of the proof, we did not use mean convexity, hence above conclusion is still valid under our alternative assumptions \eqref{alternative assumption on initial data cylindrical flows}. Then since we know that the support of a $F$-stationary varifold $V(-1)$ is equal to $M^{n,k}_{-1}$, which is a connected, smooth, embedded hypersurface, by constancy theorem for $F$-stationary integral varifolds (see section 8.4 in \cite{simon1984lectures}), we see that 
\begin{equation}
    \mu_{-1}= m\mathcal{H}^n|_{M^{n,k}_{-1}}
\end{equation}
for some positive integer $m$. The entropy bound \eqref{entropy bound} immediately implies that $m = 1$, i.e $\mu_t$ has multiplicity one. Local regularity theorem \cite{White2005-nl} implies that the convergence \eqref{tangent flow at infinity alternative} is actually smooth, meaning $M^{\infty}_t$ is an asymptotically $(n,k)$-cylindrical flow. Then by following the remaining proof of theorem \ref{dynamical stability of cylindrical translators intro}, we obtain corollary \ref{alternative dynamical stability}.
\end{proof}
\section*{Acknowledgment}
The author would like to thank his advisor Prof. Nata\v sa \v Se\v sum for her various suggestions and comments to improve the results of this paper.
\section*{Statements and Declarations}
\begin{itemize}
    \item Competing interests : The author has no competing interests to declare that are
relevant to the content of this article.
\end{itemize}
\section*{Data availability statements}
Data sharing not applicable to this article as no datasets were generated or analyzed
during the current study
\printbibliography
\end{document}